\documentclass[12pt,a4paper]{article}
\usepackage{amsfonts, amsthm}
\usepackage{hyperref}
\usepackage[english]{babel}
\usepackage{amsmath}
\usepackage{amssymb}
\usepackage{url}
\usepackage{amsxtra}
\usepackage{mathrsfs}
\usepackage{fancyhdr}
\usepackage{enumerate}
\usepackage{graphicx}
\usepackage{pdfsync}
\usepackage{psfrag}
\usepackage{tikz}
\usepackage{soul}

\newtheorem{theorem}{Theorem}[section]
\newtheorem{proposition}[theorem]{Proposition}
\newtheorem{lemma}[theorem]{Lemma}

\newtheorem{remark}{Remark}[section]

\theoremstyle{definition}

\numberwithin{equation}{section}

\newcommand{\R}{\mathbb{R}}

\newcommand{\lap}{\bigtriangleup}

\renewcommand{\Re}{\mbox{Re}}
\newcommand{\D}{\mathcal D}

\newcommand{\ent}[1]{\lfloor #1\rfloor}

\newcommand{\an}[1]{\langle #1 \rangle}

\newcommand{\grad}{\bigtriangledown}
\newcommand{\Tk}{\mathcal{T}_{q_k,\zeta}}
\newcommand{\Jac}{\textrm{Jac}}

\begin{document}

\author{Federico Cacciafesta\footnote{Dipartimento di Matematica, Universit$\grave{\text{a}}$ degli studi di Padova,  Via Trieste, 63, 35131 Padova PD - Italy. {\em email:} cacciafe@math.unipd.it}, Anne-Sophie de Suzzoni\footnote{Universit\'e Paris 13, Sorbonne Paris Cit\'e, LAGA, CNRS ( UMR 7539), 99, avenue Jean-Baptiste Cl\'ement, F-93430 Villetaneuse, France. {\em email:} adesuzzo@math.univ-paris13.fr}, Diego Noja\footnote{Dipartimento di Matematica e Applicazioni, Universit\`a di Milano Bicocca,  via R. Cozzi, 55, 20125 Milano MI - Italy. {\em email:} diego.noja@unimib.it } }

\title{A Dirac field interacting with point nuclear dynamics}

\maketitle

\begin{abstract}
The system describing a single Dirac electron field coupled with classically moving point nuclei is presented and studied. The model is a semi-relativistic extension of corresponding time-dependent one-body Hartree-Fock equation coupled with classical nuclear dynamics, already known and studied both in quantum chemistry and in rigorous mathematical literature. We prove local existence of solutions for data in $H^\sigma$ with $\sigma\in[0,\frac32[$. In the course of the analysis a second new result of independent interest is discussed and proved, namely the construction of the propagator for the Dirac operator with several moving Coulomb singularities.
\end{abstract}

\section{Introduction}
The analysis of complex atomic matter behavior starting from first principles is nowadays a well developed subject, with a rich literature both on the theoretical and computational sides. In low energy regime there is often a good or excellent agreement between theoretical description and experimental results. Things are different in presence of heavy atoms, where relativistic contributions become essential to reliable calculations of spectral and other relevant properties of the involved systems. Notwithstanding the existence of several partly efficient computational strategies, the understanding of the subject from a theoretical and rigorous point of view is at present rather poor. This is due to the lack of a consistent many body Dirac theory, in contrast with the many body Schr\"odinger theory so successful in the non relativistic regime. It is not even clear, to give a basic example, the behavior of the system composed by two Dirac particles interacting via a Coulomb potential (see \cite{D} and the recent preprint \cite{DeckOel18}); for this elementary two particle system it is widely believed but not proved that the essential spectrum is given by the whole real line and there are no eigenvalues.\\ A possible and perhaps unavoidable way out of the difficulties caused by the spectral obstruction to a many body Dirac theory consists in resorting on Quantum Electrodynamics  to obtain an effective theory. This suggestive program, if theoretically satisfying and promising, is however at present far from being fully developed.\\ In view of this incomplete and uncertain state of affairs, in this paper we want to follow a less ambitious but however non trivial goal, that is to give the local well posedness for the dynamics of a single Dirac electron field interacting with nuclear matter described, as often in Quantum Chemistry, as $N$ moving point classical particles.\\
Namely, we will study the Cauchy problem
\begin{equation}\label{sist2}
\begin{cases}
\displaystyle
i\frac{\partial u}{\partial t}(t,x)=(\mathcal{D}+\beta)u(t,x)-\sum_{k=1}^N \frac{Z_k}{|x-q_k(t)|}u(t,x)+\left(|u|^2*\frac1{|x|}\right)(t,x)u(t,x),
\\
\displaystyle
m_k\frac{d^2q_k}{dt^2}(t)=-\nabla_{q_k} W_q(t)
\\
\displaystyle
u(0,\cdot)=u_0,\qquad q_k(0)=a_k,\qquad \frac{dq_k}{dt}(0)=b_k
\end{cases}
\end{equation} 
with $N\geq1$, where 
$$
W_q(t) = -\sum_{k=1}^N Z_k\an{u, \frac1{|x-q_k|}u} + \sum_{k\neq l} Z_kZ_l \frac1{|q_k - q_l|}.
$$
\noindent
We are considering an electron with unit mass; the units are chosen in order to have $\hbar=1$ and $c=1$. Here, $ \mathcal{D}+\beta$  represents the massive $3D$ Dirac operator; we recall that $\mathcal{D}$ is defined as $\displaystyle\D=i^{-1}\sum_{j=1}^3\alpha_j\partial_j$ where the $4\times4$ Dirac matrices are given by 
\begin{equation*}
\alpha_k=\left(\begin{array}{cc}0 & \sigma_k \\\sigma_k & 0\end{array}\right),\qquad
\beta=\left(\begin{array}{cc}I_2 & 0 \\0 & -I_2\end{array}\right)
\end{equation*}
and the $\sigma_k$ ($k=1,2,3$) are the Pauli matrices, given by
\begin{equation}
\sigma_1=\left(\begin{array}{cc}0 & 1 \\1 & 0\end{array}\right),\quad
\sigma_2=\left(\begin{array}{cc}0 &-i \\i & 0\end{array}\right),\quad
\sigma_3=\left(\begin{array}{cc}1 & 0\\0 & -1\end{array}\right).
\end{equation}
\vskip5pt
We briefly discuss the model, referring to \cite{LebLio} for a comprehensive account of the subject of rigorous derivation of atomic and molecular systems, including relativistic effects and to \cite{ELS08} for details on variational techniques in stationary problems.\\
The above system contains a partial differential equation of Dirac type representing the (single) electron cloud dynamics, coupled with ordinary differential equations ruling the motion of the nuclei. The latter are described as classical point particles. The coupling shows up in two different terms in the equations: firstly the Coulomb potential evaluated at the positions of the moving nuclei appears in the Dirac equation, and then in each classical Newton equation besides the inter-particle Coulomb interaction, a further force term containing the Dirac field is present. This force term represents the Coulomb potential (at the nucleus position) due to the charge density $u^*u$ associated to the electron field. A contribution of this kind is heuristically expected on the basis of so-called Hellman-Feynman's Theorem, and it is the analogue of a similar term in the non relativistic theory of atoms and molecules. 
Finally, we add to the Dirac equation a nonlinearity of Hartree type. In the Schr\"odinger theory the Hartree nonlinearity is an effect of a reduction from a many body theory, but we do not attempt here any theoretical justification of this term and we retain this contribution by pure analogy; we only mention that in the context of Dirac-Maxwell theory, this term appears naturally in the absence of magnetic field as a by-product of the decoupling of the equations (see \cite{chadglass76}).
For relevant rigorous results regarding well posedness of the analogous model in the Schr\"odinger setting we refer to \cite{canleb}, to the above \cite{LebLio} and to references therein.\\
We discuss briefly a last issue related to the choice of the model. In a completely relativistic model the classical nuclei should have a relativistic kinematics, and the (classical) electromagnetic potentials, including the magnetic vector potential, should solve the Maxwell equations, or the wave equation in a suitable gauge. In the semi-relativistic model presented here the dynamics of the heavy nuclei is consistent, at a first approximation, with the consideration of the instantaneous Coulomb potential only, while magnetic and retardation effects are neglected. The completely non relativistic analogous of system \eqref{sist2} has already been object of study in literature (see \cite{canleb}, \cite{Baud05}, \cite{Baud06} and is known to be globally well posed; for a related paper on nonlinear Schr\"odinger equations with moving Coulomb singularities see also \cite{ozayos}). We stress the fact that, in contrast with the Schr\"odinger case, a severe difficulty here is represented by the strong 
singularity produced by the moving nuclei: indeed, the Coulomb potential exhibits the same homogeneity of the (massless) Dirac operator or, in other words, it is critical with respect to the natural scaling of the operator. This is the source of several problems, especially from the point of view of dispersive dynamics: indeed, it is not known whether Strichartz estimates hold for the flow $e^{it{(\mathcal{D}+m_e\beta+\frac\nu{|x|}})}$, even in the case $m_e=0$ (we mention the papers \cite{cacser} in which a family of local smoothing estimates for such a flow is 
proved and \cite{cacfan} in which the same result is obtained in the case of Aharonov-Bohm fields), while it is interesting to notice that for the scaling critical non-relativistic counterpart, i.e. the Schr\"odinger equation with inverse square potential, the dispersive dynamics is now completely understood (see \cite{burplan} for Strichartz, and \cite{fanfel} for time-decay estimates in even more general settings). Related analysis of dispersive estimates for perturbed Dirac flows with applications to nonlinear models can be found in \cite{escovega}, \cite{cacmag}, \cite{caccub} and references therein. Anyway, in the present paper we bypass the lack of effective dispersive estimates, adopting the strategy already working in \cite{canleb}, which is essentially an application of Segal Theorem \cite{segal}. The key ingredients will be the construction of a $2$-parameters propagator associated to the time-dependent Hamiltonian, the fact that the non local Hartree term is Lipschitz continuous and finally a two 
stage fixed  point argument. This allows to prove at least existence of {\em local} solutions for system \eqref{sist2}.\\ As a final remark on the model, notice that a complete particle-field system should include the coupling with the electromagnetic field, and so Maxwell equations with sources. For a point particle this presumably entails serious difficulties, which in the case of the completely classical Maxwell-Lorentz system are well known and unsolved (see for example \cite{spohn} for a complete discussion and \cite{NP1, NP2} and references therein for a rigorous treatment of a related model).
\vskip2pt
\noindent
Before stating our main results, let us fix some useful notations.\\
{\bf Notations.} 
With $L^2$ and $H^\sigma$ we will denote the Lebesgue and Sobolev spaces $L^2(\mathbb{R}^3,\mathbb{C}^4)$ and $H^\sigma(\mathbb{R}^3,\mathbb{C}^4)$ respectively.\\
With $\mathcal{L}(X,Y)$ we will denote the space of bounded linear operators $A:X\rightarrow Y$ and we pose $\mathcal L(X)\equiv \mathcal L(X,X)\ .$\\
By the symbol $1_{N\geq 2}$ we mean the number $1$ if we are in the case of several nuclei, the number $0$ in the case of a single nucleus.\\
For the sake of brevity, when no ambiguity is possible, we will often omit the dependence $t\in[0,T]$ and $k=1,\dots N\ $ in writing expression such as $\displaystyle\sup_{k,t}$.
\\\\
\vskip2pt
\noindent
Our first result consists in showing existence of a two-parameters propagator associated to the time dependent Dirac Hamiltonian with moving Coulomb singularities
\begin{equation}\label{ndimham}
H(t)=\mathcal{D}+\beta+V(x,t) 
\end{equation}
where $\displaystyle V(x,t)=-\sum_{k=1}^N\frac{Z_k}{|x-q_k(t)|}$.\\
We will assume to be, when $N\geq 2$, in a {\em no-collision setting}; namely, we will require the initial positions $a_k$ to be well-separated, together with a bound on the velocities $\dot{q}_k(t)$ (hypotheses \ref{sepin} and first of \ref{assteo1}). 
\noindent
\begin{theorem}\label{teo1}
Assume that $|Z_k|<\frac{\sqrt3}2$ $\forall k$ and that
\begin{equation}\label{sepin}
\min \{|q_k(0)-q_l(0)| \; | \; k\neq l\} = 8\varepsilon_0
\end{equation} for some $\varepsilon_0>0$. Then if the nuclei trajectories $q_1,\cdots , q_N$ are $W^{2,1}([0,T])$ and there exists $C_{\dot q}$ (independent from $T$) such that
\begin{equation}\label{assteo1}
(1+T1_{N\geq 2})\sup_{k} \|\dot{q}_k\|_{L^\infty([0,T])}\leq C_{\dot q},\qquad \sup_{k} \|\ddot{q}_k(t)\|_{L^1([0,T])} <\infty
\end{equation} 
for some $T>0$, then the flow of the equation
$$
i\partial_t u = H(t) u
$$
is well defined and given by a family of operators $U_q(t,s)$ with 
$$U_q \in \mathcal C([0,T]^2,\mathcal L(H^\sigma))$$
for any $\sigma\in[0,\frac32)$ with $H^\sigma\rightarrow H^\sigma$ norms uniformly bounded in $t$, $s$, and $q$; $U_q(t,s)$ satisfies
$$
U_q(t,s) \circ U_q(s,r) = U_q(t,r), \ \ \ U_q(t,t)=\mathbb{I}\ ,
$$ 
$$
i\partial_t U_q(t,s) = H(t)U_q(t,s) \; , \; i\partial_s U_q(t,s) = - U_q(t,s)H(s)\ .
$$
Moreover if $q^{(1)} = (q_1^{(1)},\hdots ,q_N^{(1)})$ and $q^{(2)} = (q_1^{(2)},\hdots ,q_N^{(2)})$ belong to $\mathcal C^2([0,T])$, satisfy the above hypotheses and assuming  moreover that $q^{(1)}(0) = q^{(2)}(0) = (a_1,\hdots ,a_N)$, then there exists $C$ (independent from $T$) such that for all $t,s \in [0,T]^2$, we have 
$$
\|U_{q^{(1)}}(t,s)- U_{q^{(2)}}(t,s)\|_{H^{\sigma}\rightarrow H^{\sigma-1}} \leq  C \sup_{k} (1+1_{N\geq 2}T) \|(\dot{q_k}^{(1)})- (\dot{q_k}^{(2)})\|_{L^\infty}
$$
for any $\sigma\in[1,\frac32)$.
\end{theorem}

\begin{remark}\label{selfadjrk}
We recall that
\begin{equation}\label{diraccoul}
\mathcal{D}+\beta\pm\frac Z{|x|}
\end{equation}
(and variants thereof) is a self-adjoint operator for $|Z|<\frac{\sqrt3}2$ on $H^1$ (see \cite{schmincke}); a distinguished  self-adjoint extension can actually be built also in the wider range $|Z|\leq 1$ (see \cite{estebanloss} and references therein; see also the recent review \cite{Gallone}). Similar results hold in the case of multi-centric Coulomb potentials but for atomic numbers $|Z_k|<\frac{\sqrt3}2$ (Levitan-Otelbaev theorem,  \cite{LO,klaus}), and we retain this condition along the paper without other comments. 

\end{remark}

\begin{remark}
By applying a standard continuation argument, the propagator $U_q(t,s)$ defined in Theorem \ref{teo1} can be extended to global times provided one assumes, instead of \eqref{assteo1}, global bounds on $|\dot{q}_k(t)|$, $|\ddot{q}_k(t)|$. However, in the subsequent analysis of the coupled dynamics defined by system \eqref{sist2}, these two terms can be bounded only locally in time, due to the absence of positive definite conserved quantities. Therefore we will only be able to obtain a local evolution in the coupled system.
\end{remark}

\begin{remark}
The proof of Theorem \ref{teo1} borrows ideas from \cite{katoyajim}, in which the authors prove a similar result in the case of the Dirac equation perturbed by the retarded Lienard-Wiechert potentials produced by relativistically moving nuclei. The main technical tool consists there in introducing a local Lorentz transformation depending on the particle trajectories which simultaneously freezes the position of the moving singularities, in such a way to change from moving Coulomb singularities to stationary ones, and then to resort on classical theory of self-adjointness for perturbations of the Dirac operator. 
We stress however that the difference in the model (the Lienard-Wiechert potentials discussed in \cite{katoyajim} are significantly more involved and require stronger assumptions) and our need of detailed estimates for subsequent analysis do not allow a reduction of the present result to the one in \cite{katoyajim}.
\end{remark}

%\textcolor{red}
%{
\begin{remark}\label{structs}
The threshold $\sigma=\frac32$ in Theorem \ref{teo1} seems to be structural in the following sense: in order to apply Kato's results for the construction of the propagator associated to a time-dependent Hamiltonian $H(t)$, one of the (sufficient) conditions is $H(t)\in C([0,T]; \mathcal{L}(Y,X))$. Therefore, beyond the regularity in time, which will be guaranteed by assumptions  \eqref{assteo1}, one needs the Hamiltonian to be bounded from $H^{\sigma+1}$ into $H^{\sigma}$, and this fact will be true, as a consequence of generalized Hardy-Rellich inequality, only for $\sigma$ up to $\frac12$. Moreover, we stress the fact that the space $H^\frac32$ appears as a natural threshold of regularity for the eigenstates of the Dirac-Coulomb operator (see the Appendix), and this seems to confirm the optimality of our result.
\end{remark}
%}

Then we study local well posedness for the Cauchy problem \eqref{sist}. For the sake of simplicity, we state this result in the simplified framework of a single nucleus first. In this case, the system writes (we consider, without loss of generality, the initial condition $q(0)=0$)
\begin{equation}\label{sist}
\begin{cases}
\displaystyle
i\frac{\partial u}{\partial t}(t,x)=(\mathcal{D}+\beta)u(t,x)-\frac{Z}{|x-q(t)|}u(t,x)+\left(|u|^2*\frac1{|x|}\right)(t,x)u(t,x),
\\
\displaystyle
m\frac{d^2q}{dt^2}(t)=\langle u(t) | \nabla \frac{Z}{|\cdot-q(t)|}|u(t)\rangle
\\
\displaystyle
u(0,\cdot)=u_0,\qquad q(0)=0,\qquad \frac{dq}{dt}(0)=\dot{	q}_0.
\end{cases}
\end{equation}
We are here using the Dirac bra-ket notation to denote 
\begin{equation}\label{nlclass}
\langle u(t) | \nabla \frac{Z}{|x-q(t)|}|u(t)\rangle=-Z\int_{\mathbb{R}^3} \langle u(t),u(t)\rangle_{\mathbb{C}^4}\frac{x-q(t)}{|x-q(t)|^3}.
\end{equation}
We get the following result.

\begin{theorem}\label{teo2}
Let $|Z|< \frac{\sqrt3}2$ and $\sigma\in[1,3/2)$. There exist $C_1$ and $C_2$ depending on $Z$ and $m$, such that for all $R\in \R_+$, all $u_0 \in H^\sigma$ such that $ \|u_0\|_{H^\sigma} \leq R$, and all initial conditions $q_0$ such that $|\dot q_0 |\leq C_1$,
%\begin{itemize}
%\itemif $\sigma\in(1,3/2)$, 
then system \eqref{sist} admits a solution in $\mathcal C([0,T], H^\sigma(\R^3)) \times \mathcal C^2([0,T],\R^3)$, for $T = \frac1{C_2(1+R^2)}$;
%\item
%if $\sigma\in(3/2,2]$, then system \eqref{sist} is well posed in $\mathcal C([0,T], H^\sigma(\R^3)) \times \mathcal C^2([0,T],\R^3)$, for $T = \frac1{C_2(1+R^2)}$, that is it admits a unique solution and the flow is continuous in the initial data.
%\end{itemize}
\end{theorem}

The analogous result in the multi-nuclear case is then the following

\begin{theorem}\label{teo3}
Let $N\geq2$, $|Z_k|< \frac{\sqrt 3}{2}$ for all $k=1,\dots N$ and $\sigma\in[1,3/2)$. There exist $C_1$ and $C_2$ depending on $(Z_k)_k$ and $(m_k)_k$ and $\varepsilon_0$, such that for all $R\in \R_+$, all $u_0 \in H^\sigma$ such that $ \|u_0\|_{H^\sigma} \leq R$, and all initial conditions $(q_k(0))_k$ satisfying \eqref{sepin} and all vectors $\dot{q}_0$ such that $|\dot q_0 |\leq C_1$,
%\begin{itemize}
%\itemif  $\sigma\in(1,3/2)$,
 then system \eqref{sist} admits a solution in $\mathcal C([0,T], H^\sigma(\R^3)) \times \mathcal C^2([0,T],\R^3)$, for $T = \frac1{C_2(1+R^2)}$;
%\item
%if $\sigma\in(3/2,2]$, then system \eqref{sist} is well posed in $\mathcal C([0,T], H^\sigma(\R^3)) \times \mathcal C^2([0,T],\R^3)$, for $T = \frac1{C_2(1+R^2)}$,  that is it admits a unique solution and the flow is continuous in the initial data.
%\end{itemize}
\end{theorem}

We give some remarks on these results.

%\textcolor{red}
%{
\begin{remark}
The threshold $\sigma=\frac32$ is a consequence of Theorem \ref{teo1}, as one needs the two parameter propagator constructed before in order to prove existence of a solution. We moreover notice that this threshold appears again in the nonlinear model: indeed, in order to prove existence of solutions for equation \eqref{sist2} one needs to prove regularity properties for the nonlinear term \eqref{nlclass} which, as a matter of fact, turns to be H\"older continuous for $u\in H^\sigma$ for $\sigma<\frac32$, and would become Lipschitz continuous for $\sigma>\frac32$. Both these facts are a consequence, again, of generalized Hardy inequality. This means that if one was able to construct some propagator on $H^\sigma$ for $\sigma>\frac32$, this would yield not only existence of solutions for system \ref{sist2}, but also {\em local well-posedness} (i.e. uniqueness and continuous dependance on the initial data). Unfortunately, as discussed in Remark \ref{structs}, this threshold seems to be structural. To overcome these issues one can think to modify the setting of the problem working with weighted Sobolev spaces or to modify the model regularizing the singularity in the Coulomb potential. These developments will be the object of future work.
\end{remark}
%}

\begin{remark}\label{wrk}
If we look only at the nonlinear Dirac equation, our $H^\sigma$ assumption on the regularity of the initial condition is well above the critical threshold required by the scaling (see e.g. \cite{scatt}, in which the authors study global well-posedness and scattering for the Dirac equation with a nonlocal nonlinear term of the form $F(u)=(|x|^a*|u|^{p-1})u$ relying on Strichartz estimates). Nonetheless, our high regularity requirement seems to be unavoidable if one wants to deal with the classical Newtonian dynamics for the nuclei. Moreover, let us point out that the coupled system \eqref{sist2} does not exhibit any scaling law even in the case of massless electrons.
\end{remark}

\begin{remark}
It is interesting to compare Theorem \ref{teo2} with its non relativistic counterpart, i.e. Theorem 1 in \cite{canleb}. In that case the authors prove global well posedness for the Cauchy problem (i.e. the existence of a solution for any time $t>0$) for initial data in the space $H^2$; to do this, they first prove local well posedness and then extend the solution using energy conservation of the system. This strategy does not work for the Dirac equation (and thus in the present contest), as the associated energy is not positive, and therefore cannot be used to control any $H^\sigma$ norm. This is the ultimate reason why we are only able to obtain local well posedness for system \eqref{sist}. 
\end{remark}

\begin{remark} All the constants in Theorem \ref{teo3} may depend on $\varepsilon_0$. In particular, the time of existence should behave like $\varepsilon_0^\gamma$ for some $\gamma > 0$. The power $\gamma$ that one can compute while performing the proof does not seem physically relevant, hence we did not keep track of it during the proof. This smallness on $T$ may be replaced by a smallness assumptions on $Z_k$, in this case, we can have a time of existence proportional to $\varepsilon_0$.\end{remark}
\noindent
We give a brief outline of the structure of the paper.\\ Section \ref{constprop} will be devoted to the proof of Theorem \ref{teo1}, i.e. to the construction of the $2$-parameter propagator, first in the case of a single nucleus and then with several ones. Various properties of the propagator are derived: firstly, its original definition can be extended from $L^2$ to $H^\sigma$ for $\sigma<\frac32$ as a consequence of Kato's theory on two-parameter propagators, secondarily a result of continuous dependence on the trajectory is given. The case of several nuclei is more involved but the results are analogous to the single nucleus case when suitable conditions on the trajectories, preventing particles are close to collisions, are imposed. \\ 
Section \ref{proof} will be dedicated to the proofs of Theorems \ref{teo2} and \ref{teo3}, simultaneously. We summarize the strategy as follows: the solution map is first considered as acting on the electron field $u$ for every trajectory $q$, using a contraction argument in $H^\sigma$. Then a Schauder fixed point argumentis performed on the (integrated) Newton equation for the nucleus trajectory, giving the final result. The properties of the solution map depend in a crucial way on the previously proved results for the non autonomous propagator.

\medskip
{\bf Acknowledgments.} We are grateful to prof. \'Eric S\'er\'e for having introduced us to the present problem and for several enlightening discussions on the topic, to  Matteo Gallone for discussions and comments and to Jonas Lampart for pointing out a mistake in our original argument, that led to the present version of the paper.

\section{The Dirac-Coulomb propagator with moving singularities}\label{constprop}

We will present in the next subsection the proof of Theorem \ref{teo1} in the case of a single nucleus, that will be divided in several steps, as the strategy is clearer in this case; afterwards, in subsection \ref{multi}, we shall present all the necessary modifications needed in order to deal with the case of several nuclei. 

\subsection{One nucleus}\label{onneuc}

In this case, the time-dependent Hamiltonian reads
\begin{equation}\label{1dham}
H(t) = \D +\beta- \frac{Z}{|x-q(t)|}.
\end{equation}
Throughout this subsection, we will always assume that 
\begin{equation}\label{asszone}
|Z|<\frac{\sqrt{3}}2
\end{equation}
which, as discussed in Remark \ref{selfadjrk}, ensures an essentially self-adjoint (static) Dirac-Coulomb operator. Notice the following relations
$$
\|f\|_{\dot H^1}^2 = \int_{\R^3}|\nabla f|^2 \cong \int_{\R^3} |\D f|^2 
$$
and, due to the anticommutation of the Dirac matrices,
$$
\|f\|_{H^1}^2=\int_{\R^3}|\nabla f|^2+\int_{\R^3}|f|^2\cong \int_{\R^3} |(\D+\beta)f|^2.
$$

We split the proof into several Lemmas.

\begin{lemma}\label{lem-changevar} Let $T \in \R^+$, $u \in \mathcal C([0,T],  H^\sigma)\cap \mathcal C^1([0,T],H^{\sigma-1)}$ for $\sigma\geq 1$, $q\in \mathcal C^1([0,T], \R^3)$ and set  $v(t,x) = u(t, x+ q(t))$.Then $u$ solves 
\begin{equation}\label{origeq}
i\partial_t u = H(t) u
\end{equation}
if and only if $v$ solves
\begin{equation}\label{modifeq}
i\partial_t v  = H_1(t) v
\end{equation}
where $\displaystyle H_1(t) = \D + \beta- \frac{Z}{|x|} +i \dot{q}(t) \cdot \nabla$.
\end{lemma}

\begin{proof} Straightforward computation. \end{proof}

\begin{lemma}\label{lem-pert} 
Let $T\in\mathbb{R}^+$. Assume \eqref{asszone} and that
\begin{equation}\label{smallvelone}
\sup_{t\in[0,T]}|\dot{q}(t)|\leq R_Z
\end{equation}
for some suitably small constant $R_Z$. Then there exists a constant $C$ such that for all $t\in[0,T]$
\begin{equation}\label{isoiso}
\frac1C\|f\|_{H^1}\leq \|H_1(t)f\|_{L^2}\leq C\|f\|_{H^1}.
\end{equation}
In particular, for every $t\in[0,T]$, $H_1(t)$ is an isomorphism from $H^1$ to $L^2$.

%\textcolor{red}{Moreover, for any $\sigma\in(0,\frac12)$, the operator $H_1(t)$ is bounded from $H^{\sigma+1}$ into $H^\sigma$.}
\end{lemma}
\begin{proof}
Inequality \eqref{isoiso} is a consequence of general theory for Dirac-Coulomb operator. In particular, Theorem 4.4 in \cite{thaller} ensures indeed that the domain of the operator $\D + \beta- \frac{Z}{|x|}$ (i.e. the "static" Dirac-Coulomb operator) coincides with the domain of the free Dirac operator, which is $H^1$. Moreover, it is known that $0$ is not in the spectrum of $\D + \beta- \frac{Z}{|x|}$ therefore $\D + \beta- \frac{Z}{|x|}$ is an isomorphism between its domain with $H^1$ topology and $L^2$. Indeed it is a bijective linear map from the domain into $L^2$, continuous for the topologies $H^1$ and $L^2$, so by the open mapping principle its inverse $(\D + \beta- \frac{Z}{|x|})^{-1}$ is continuous for the $L^2$ and $H^1$ topology. As then $H_1(t)=\D + \beta- \frac{Z}{|x|}+i\dot{q}(t).\nabla$, 
under our assumption on $|\dot{q}(t)|$ the additional term can be treated as a (bounded) perturbation, and then \eqref{isoiso} holds.

\end{proof}

Lemma \ref{lem-pert} will be enough to construct, through general Kato's Theory, the two parameters propagator associated to the Hamiltonian $H_1(t)$ on the space $H^1$; anyway, we are actually able to extend this propagator to some higher Sobolev spaces, namely for any $\sigma<3/2$; we stress the fact that we need to go above the regularity of the domain of the Dirac-Coulomb operator, which is $H^1$, and thus this step requires some additional work.

We start by providing some fundamental functional inequalities. The first one, is the following generalized Hardy inequality.

\begin{proposition}\label{genhardy}
For any $\sigma\in[0,\frac{d}2)$ there exists a constant $C$ such that for any $f\in \dot{H}^\sigma(\R^d)$ 
\begin{equation}\label{hardgen}
\int_{\R^d}\frac{|f(x)|^2}{|x|^{2\sigma}}dx\leq C\|f\|_{\dot{H}^\sigma}.
\end{equation}
\end{proposition}

\begin{proof}
See Theorem 2.57 in \cite{bah}.
\end{proof}

Then, we need the following Rellich inequality (\cite{Rellich1, Rellich2}).

\begin{proposition}\label{lemcrucin}
Let $u\in C^\infty_c(\R^3\backslash\{0\})$. Then 
\begin{equation}\label{crucialest}
\int_{\R^3}\frac{|u|^2}{|x|^4}dx\leq C\int_{\R^3}|\Delta u|^2dx
\end{equation}
and, as a consequence, 
\begin{equation}\label{H2est}
\left\|\frac{u}{|x|}\right\|_{H^1}\leq C \|u\|_{H^2}.
\end{equation}

\end{proposition}

\begin{remark}
As a matter of fact the inequality was proved (in the radial case) by Rellich himself exactly in this form and with the optimal constant $C=(\frac{3}{4})^2$; the extension to the non radial case, here inspired by \cite{cazacu}, appears probably in several other papers but we were not able to find a reference and so we include a proof for the sake of completeness. We stress the fact that the inequality is usually stated and proved on functions defined on the wholes space $\R^n$ but in higher dimensions $n$.
\end{remark}

\begin{proof}
We are not interested in providing optimal constant in the inequality, and so along the proof any constant appearing will be all indicated by a generic $C$ that will be allowed to change from line to line. We use the well known spherical harmonics decomposition to write
$$
u(x)=\sum_{k=0}^\infty u_k(r)\phi_k(\theta)
$$
where $\phi_k(\theta)\in L^2(S^2)$ for $k\geq0$ are the spherical harmonics of degree $k$ (which might be assumed to be $L^2$-normalized) which, we recall, satisfy the property $\Delta_{S^2}\phi_k=c_k\phi_k$ with $c_k=k(k+1)$. We then get, as the action of the laplacian in spherical coordinate is given by $\Delta=\partial_{rr}++\frac2r\partial_r+\frac1{r^2}\Delta_{S^2}$, that we can write
$$
\int_{\R^3}|\Delta u|^2dx=\sum_{k=0}^\infty\int_0^{+\infty}\left(|\partial_{rr}u_k|^2+\frac{c_k^2}{r^4}u_k^2-\frac{2c_k}{r^2}u_k\partial_{rr}u_k\right)r^2dr
$$
that, after integrating by parts, becomes
\begin{equation}\label{radla}
\int_{\R^3}|\Delta u|^2dx=
\end{equation}
\begin{equation*}\sum_{k=0}^\infty\left(\int_0^{+\infty}|\partial_{rr}u_k|^2r^2dr+2(c_k+1)\int_0^{+\infty}|\partial_ru_k|^2dr+c_k(c_k-1)\int_0^{+\infty}\frac{u^2_k}{r^2}dr\right).
\end{equation*}
Notice that $c_k(c_k-1)\geq0$ for every $k\geq0$. 
We now use the following estimates:
\begin{equation*}
\int_0^{+\infty}|\partial_{rr}u_k|^2r^2dr\geq C\int_0^{+\infty}\frac{|\partial_{r}u_k|^2}{r^2}r^2dr,\qquad \forall k\geq0
\end{equation*}
which is a direct application of Hardy inequality for radial functions, and
\begin{equation*}
\int_0^{+\infty}|\partial_{r}u_k|^2dr\geq C\int_0^{+\infty}\frac{|u_k|}{r^4}^2r^2dr,\qquad \forall k\geq0,
\end{equation*}
which is a consequence of integration by parts and Cauchy-Schwarz. Plugging these estimates into \eqref{radla} we thus obtain
\begin{equation*}
\int_{\R^3}|\Delta u|^2dx\geq C\sum_{k=0}^\infty \int_0^{+\infty}\frac{|u_k|^2}{r^4}r^2dr
\end{equation*}
that is exactly \eqref{crucialest}.

Estimate \eqref{H2est} comes as a direct consequence of \eqref{crucialest} and standard Hardy's inequality: indeed, for any $u\in C^\infty_c(\R^3\backslash\{0\})$ we have
\begin{equation}\label{H2est}
\left\|\frac{u}{|x|}\right\|_{H^1}\lesssim 
\left\|\frac{\nabla u}{|x|}\right\|_{L^2}+
\left\|\frac{u}{|x|^2}\right\|_{L^2}\lesssim \|u\|_{H^2}.
\end{equation}
\end{proof}

The following lemma represents the key step in order to extend the propagator to higher order fractional Sobolev spaces.

\begin{lemma}\label{hardyrell}
For any $t>0$, the Hamiltonian $H_1(t)$ is a bounded operator from $H^{\sigma}$ into $H^{\sigma-1}$
 for any $\sigma\in[1,3/2)$. 
 \end{lemma}

\begin{proof}
We need to show that the operator $\frac1{|x|}$ is bounded from $H^{\sigma}$ into $H^{\sigma-1}$ for $\sigma\in[1,3/2)$, that is inequality
\begin{equation}\label{cry}
\left\| \frac{u}{|x|}\right\|_{H^{\sigma-1}}\leq \|u\|_{H^{\sigma}},\qquad \sigma\in[1,\frac32),
\end{equation}
as the boundedness of the other terms is trivial. In order to prove this fact, we introduce the regularizing potential $\frac1{\sqrt{|x|^2+\varepsilon^2}}$, with $\varepsilon>0$. First, we have the following inequality, for any $u\in C^\infty_c(\R^3)$
\begin{equation}\label{in1}
\left\|\frac{u}{\sqrt{|x|^2+\varepsilon^2}}\right\|_{L^2}\lesssim \|u\|_{H^1},
\end{equation}
that is a simple consequence of Hardy's inequality as $\frac1{\sqrt{|x|^2+\varepsilon^2}}\leq\frac1{|x|}$. Then, we also have, still for any $u\in C^\infty_c(\R^3)$
\begin{equation}\label{in2}
\left\|\frac{u}{\sqrt{|x|^2+\varepsilon^2}}\right\|_{H^1}\lesssim \|u\|_{H^2}.
\end{equation}
In order to prove this inequality, we write
\begin{equation*}
\left\|\frac{u}{\sqrt{|x|^2+\varepsilon^2}}\right\|_{H^1}\lesssim
\left\|\frac{\nabla u}{\sqrt{|x|^2+\varepsilon^2}}\right\|_{L^2}+
\left\|u\nabla\left(\frac{1}{\sqrt{|x|^2+\varepsilon^2}}\right)\right\|_{L^2}=I+II.
\end{equation*}
For $I$ we easily have from Hardy's inequality
$$
I\leq \left\|\frac{\nabla u}{|x|}\right\|_{L^2}\leq \|u\|_{H^2}.
$$
We now deal with $II$: first of all, notice that we have
$$
\left|\nabla\left(\frac{1}{\sqrt{|x|^2+\varepsilon^2}}\right)\right|\leq \frac1{|x|^2+\varepsilon^2}.
$$
Now, let us introduce two functions $\chi$ and $\eta$ such that $f=f\eta+f\chi$, with supp$(\eta)\subset B(0,\tilde{\varepsilon})$, supp$(\chi)\subset B^c(0,\tilde{\varepsilon}/2)$ with $\tilde{\varepsilon}$ to be fixed later, and $\eta+\chi=1$. We thus write
$$
II^2\leq\int\frac{|u|^2}{(|x|^2+\varepsilon^2)^2}\leq 
\int\frac{|\chi u|^2}{(|x|^2+\varepsilon^2)^2}+
\int\frac{|\eta u|^2}{(|x|^2+\varepsilon^2)^2}=II_1^2+II_2^2.
$$
The term $II_1^2$ gives no problem, as the support of $\chi$ allows to neglect the singularity in zero, and one simply has from \eqref{crucialest}
$$
II^2_1\leq \int\frac{|u|^2}{|x|^4}\leq \|u\|_{H^2}^2.
$$
On the other hand, for $II^2_2$ we can write
$$
II^2_2\leq \int_{B(0,\tilde{\varepsilon})}\frac{|u|^2}{\varepsilon^4}\lesssim \|u\|_{L^\infty}^2\frac{\tilde{\varepsilon}^3}{\varepsilon^4}\leq \|u\|_{H^2}
$$
where in the last inequality we have chosen $\tilde{\varepsilon}=\varepsilon^{4/3}$. This concludes the proof of \eqref{in2}. We can now interpolate between inequalities \eqref{in1} and \eqref{in2} to obtain the folllowing family of inequalities for any $u\in C^\infty_c(\R^3)$: for any $\varepsilon>0$ and any $\sigma\in [1,2]$ we have
\begin{equation}\label{intin}
\left\| \frac{u}{\sqrt{|x|^2+\varepsilon^2}}\right\|_{H^{\sigma-1}}\leq \|u\|_{H^{\sigma}}.
\end{equation}
We now want to send $\varepsilon\rightarrow 0$ to retrieve our result: first, we note that  
\begin{equation}\label{L2conv}
\left\|\left(\frac{1}{\sqrt{|x|^2+\varepsilon^2}}-\frac1{|x|}\right)u\right\|_{L^2}\rightarrow 0
\end{equation}
for $\varepsilon\rightarrow 0$. More precisely, we have that
$$
\left| \frac1{\sqrt{\varepsilon^2+|x|^2}}-\frac1{|x|}\right|=\left|\frac{|x|-\sqrt{\varepsilon^2+|x|^2}}{|x|\sqrt{\varepsilon^2+|x|^2}}\right|=\left|\frac{\varepsilon^2}{|x|\sqrt{\varepsilon^2+|x|^2}(|x|+\sqrt{\varepsilon^2+|x|^2})}\right|
\leq \frac{\varepsilon^2}{|x|^3}
$$ and 
$$\left| \frac1{\varepsilon+|x|}-\frac1{|x|}\right|\leq \frac{1}{|x|}$$ imply
$$
\left| \frac1{\varepsilon+|x|}-\frac1{|x|}\right|\leq \frac{\varepsilon^{s}}{|x|^{3s/2}}\left(\frac1{|x|}\right)^{1-s /2}=\frac{\varepsilon^{s}}{|x|^{1+s}}
$$
for $s\in(0,1/2)$; therefore, using \eqref{hardgen}, we get
\begin{equation}\label{thisweneed}
\left\|\left(\frac{1}{\sqrt{|x|^2+\varepsilon^2}}-\frac1{|x|}\right)\right\|_{L^2}^2\leq
\varepsilon^{2(\sigma-1)}\int\frac{|u|^2}{|x|^{\sigma}}\leq \varepsilon^{2(\sigma-1)}\|u\|_{H^{\sigma}}
\end{equation}
which in particular proves \eqref{L2conv}. 

We are now ready to prove \eqref{cry}. We use Fourier transform $\mathcal{F}$ and then bound as follows: fix $R>0$, then
\begin{equation*}
\int_0^R\langle \xi\rangle^{\sigma-1}\left|\mathcal{F}\left(\frac{u}{|x|}\right)\right|^2(\xi)d\xi
\end{equation*}
\begin{equation*}
\leq
\int_0^R \langle \xi\rangle^{\sigma-1}\left|\mathcal{F}\left(\frac{u}{|x|}\right)-\mathcal{F}\left(\frac{u}{\sqrt{\varepsilon^2+|x|^2}}\right)\right|^2(\xi)d\xi+\int_0^R\langle \xi\rangle^{\sigma-1}\left|\mathcal{F}\left(\frac{u}{\sqrt{\varepsilon^2+|x|^2}}\right)\right|^2(\xi)d\xi
\end{equation*}
\begin{equation*}
\leq
\langle R\rangle^{\sigma-1}\left\|\frac{u}{|x|}-\frac{u}{\sqrt{\varepsilon^2+|x|^2}}\right\|_{L^2}^2+
\left\|\frac{u}{\sqrt{\varepsilon^2+|x|^2}}\right\|_{H^\sigma}^2
\end{equation*}
\begin{equation*}
\leq
\|u\|_{H^{\sigma}}^2\left(1+\langle R\rangle^{\sigma-1}\varepsilon^{\sigma-1}\right)^2
\end{equation*}
where in the last inequality we have used \eqref{thisweneed} and \eqref{intin}. We now take $\varepsilon$ such that $\langle R\rangle^{\sigma-1}\varepsilon^{\sigma-1}\rightarrow 0$ to conclude
$$
\int_0^R\langle \xi\rangle^{\sigma-1}\left|\mathcal{F}\left(\frac{u}{|x|}\right)\right|^2(\xi)d\xi\leq \|u\|_{H^{\sigma}}^2
$$
that is \eqref{cry}.
\end{proof}

\begin{remark}
Notice that the tempting argument of interpolating inequality \eqref{H2est} with standard Hardy's inequality to conclude the desired boundedness of the multiplication operator $\frac1{|x|}$ from $H^{\sigma}$ into $H^{\sigma-1}$ does not come for free, as \eqref{H2est} is only proved for functions $u\in C^\infty_c(\R^3\backslash\{0\})$ which is not dense in $H^\sigma$ for $\sigma\geq \frac{1}{2}$. This fact forces our two steps interpolation procedure.
\end{remark}
%and relying on \eqref{crucialest} and standard Hardy's inequality, imply the estimate
%$$ 
%\left\|\frac{u}{|x|}\right\|_{H^1}\lesssim \|u\|_{H^2}.
%$$
%By interpolating again with Hardy inequality, we get the estimate
%\begin{equation}\label{interph}
%\left\|\frac{u}{|x|}\right\|_{H^\sigma(\R^3\backslash\{0\})}\lesssim \|u\|_{H^{\sigma+1}(\R^3\backslash\{0\})}.
%\end{equation}
%for any $\sigma\in[0,1]$. Now, as the spaces $H^\sigma(\R^3\backslash\{0\})=H^\sigma(\R^3)$ (see \cite{???} or PROVE IT) as long as $\sigma\in(0,3/2)$, estimate \eqref{interph} can be extended to the Sobolev spaces on the whole $\R^3$, and this concludes the proof.

We are now in position to state the the main result of this first part, that is the existence of a two-parameter propagator associated to the time-dependent Hamiltonian $H_1(t)$.

\begin{proposition}\label{proppropag}
Suppose the hypothesis of Lemma \ref{lem-pert} hold true and moreover that 
\begin{equation}\label{smallaccone}
\|\ddot{q}(t)\|_{L^1([0,T])}<\infty.
\end{equation}
Then there exist \\
1) a family of operators $(U_1(t,s))_{t,\sigma\in [0,T]^2}$ from $L^2$ to $L^2$ strongly continuous in $(t,s)\in [0,T]$, uniformly bounded in $(t,s)\in [0,T]\times[0,T]$ and with the properties
\begin{align*}
U_1(t,r)&=U_1(t,s) \circ U_1(s,r) U_1(t,r),\ \ \quad \ 0\leq r\leq \sigma\leq t\leq T \ ; \\
U_1(t,t)&=\mathbb{I},  \ \ \ \ \ \ \ \ \ \ \ \ \ \ \ \ \ \ \ \ \ \ \ \ \ \ \ \ \ \ \ \ \ \ 0\leq t\leq T;\\
i\partial_t U_1(t,s) &= H_1(t)U_2(t,s) \quad  \quad \quad \ \quad \quad \quad \ \ \ 0\leq \sigma\leq t\leq T;\\
i\partial_s U_1(t,s) &= - U_1(t,s)H_1(s)\quad \quad \quad \quad \quad\quad \ \ 0\leq \sigma\leq t\leq T;
\end{align*}
2) the same family of operators, again indicated with $(U_1(t,s))_{t,\sigma\in [0,T]^2}$ restrict invariantly from $H^\sigma$ to $H^\sigma$ for ant $\sigma\in[0,\frac32)$ with the same properties as above.
\end{proposition}

\begin{remark}
The condition $\ddot{q}$ above is needed in order to ensure that $H_1(t)$ is of bounded variations (in time) as an operator from $H^1$ to $L^2$. What is more, the operator $U_1(t,s)$ is uniformly bounded from $H^1$ to $H^1$ or from $L^2$ to $L^2$ in $q$ in balls of $W^{2,1}$.
\end{remark}

\begin{proof}
This result is a direct consequence of a very well known and developed theory. Indeed, Lemma \ref{lem-pert} ensures that $H_1(t)$ is a bounded, continuously differentiable in time operator from $H^{\sigma}$ to $H^{\sigma-1}$ for any $\sigma\in[1,\frac32)$. This allows to use the well known results due to Kato (see \cite{Kato1970}, \cite{Kato56} and, in particular, Theorem 2 in \cite{Kato53}; see also \cite{SchGrie14} and \cite{Schn} for recent surveys) to show that $H_1(t)$ generates a two parameter propagator and a well posed dynamics.
\end{proof}
\begin{remark}The fact that the propagator is in $L^2\rightarrow L^2$ is due to the fact that $H_1(t)$ satisfies \eqref{isoiso} and that it is of bounded variations in time. Besides already quoted literature, a good reference where restriction of the evolution family to dense subset is treated is the classic treatise of Pazy, chapter V \cite{Pazy}. For more recent exposition, survey and clarification of some of the hypotheses in original papers, we mention \cite{SchGrie14} and \cite{Schn} where by the way it is remarked that the $C^1$ property of the generator is the really relevant one. Also, we mention \cite{movdir}, in which the authors construct a propagator for a Dirac equation with a moving small potential. The meaning of the two parts of Proposition \ref{proppropag} is that equation \eqref{modifeq} is well posed in both $L^2$ and $H^\sigma$, and hence the same holds true for the original equation \eqref{origeq}, thanks to Lemma \ref{lem-changevar}.
\end{remark}

%Now, in our hypotheses on $q(t)$, we have that $H_1(t)$ is continuously differentiable in time as an operator from $H^1$ to $L^2$. This allow to use the well known results due to Kato (see \cite{Kato1970}, \cite{Kato56} and, in particular, Theorem 2 in \cite{Kato53}; see also \cite{SchGrie14} and \cite{Schn} for recent surveys) to show $H_1(t)$ generates a two parameter propagator and a well posed dynamics. For more recent exposition, survey and clarification of some of the hypotheses in original papers, we mention \cite{SchGrie14} and \cite{Schn} were by the way it is remarked that the $C^1$ property of the generator is the really relevant one.
%Adapted to the present contest, Kato theory gives the following

%The result above allows then to define, under the sole assumptions \eqref{smallvelone}-\eqref{smallaccone}, a two-parameter propagator on \textcolor{red}{$H^\sigma$ for any $\sigma\in[0,\frac32[$}. In what follows we will anyway need to go above the $H^1$ regularity: in the study of system  \eqref{sist}, we will indeed see that the classical dynamics for the nuclei will need $u$ to be at least $H^{3/2+}$. Therefore, we prove the following
%\begin{proposition}\label{moreregprop}
%Let $q$ and $Z$ satisfy \eqref{asszone},\eqref{smallvelone} and \eqref{smallaccone}. Then, $U_1$ belongs to $\mathcal C([0,T]^2, \mathcal L(H^\sigma))$ for any $\sigma\geq1$ and its norm is uniformly bounded in $q$ in balls of $W^{2,1}$.
%\end{proposition}

In order to define the propagator associated to the original equation \eqref{origeq}, one only needs to re-change variables.

\begin{proposition}\label{prop-linH1} 
Let $q$ and $Z$ satisfy \eqref{asszone}, \eqref{smallvelone} and \eqref{smallaccone}. 
 Then the flow of the equation
$$
i\partial_t u = H(t) u
$$
with $H(t)$ defined by \eqref{1dham} is given by a family of operators $U(t,s)=U_q(t,s)$ satisfying 
$$
U_q(t,s) \circ U_q(s,r) = U_q(t,r), \ \ \ U_q(t,t)=\mathbb{I}.
$$ 
and
$$
i\partial_t U_q(t,s) = H(t)U_q(t,s) \; , \; i\partial_s U_q(t,s) = - U_q(t,s)H(s)
$$
with 
$$U_q \in \mathcal C([0,T]^2,\mathcal L(H^\sigma))$$ for any $\sigma\in[0,\frac32)$. In particular the norms 
$$
\|U_q(t,s)\|_{H^\sigma\rightarrow H^\sigma}
$$
are uniformly bounded in $t$, $s$, and $q$.
\end{proposition}

\begin{proof} Let $I(t)$ be the smooth isometry of $H^\sigma$ for any $s$
\begin{equation}\label{It}
I(t) f(t,x) = f(t,x +q(t)). 
\end{equation}
The operators $U(t,s) = I(t)^{-1} U_1(t,s) I(s)$ satisfy the conclusions, as remarked in Lemma \ref{lem-changevar}.
\end{proof}

We now show the continuity in $q$ of the propagator $U$.

\begin{proposition}\label{lem-cont} The operator $U(t,s) = U_q(t,s)$ depends continuously on $q$ as an operator from $ H^{\sigma}$ to $H^{\sigma-1}$, in the sense that for all $q_1,q_2 \in \mathcal C^1([0,T])$ satisfying the assumptions of Lemma \ref{lem-pert} and all $(t,\tau) \in [0,T]^2$,

\begin{equation}\label{H2cont}
\|U_{q_1}(t,\tau)- U_{q_2}(t,\tau)\|_{H^{\sigma} \rightarrow H^{\sigma-1}} \leq C T \|\dot{q}_1 - \dot{q}_2\|_{L^\infty([0,T])} 
\end{equation}
for any $\sigma\in[1,\frac32)$, with $C$ independent from $T$ and $q$ (with our assumptions of boundedness on $q$).
\end{proposition}

%\begin{remark}
%Before proving this Proposition, let us make a comment. According to \cite{Kato53}, it appears that $U_1(t,s)$ is the limit as an operator from $ H^1$ to $L^2$ of
%\begin{equation}\label{U1seq}
%U_1(n,t,s) = e^{-i(t-K/n)H_1(K/n)}\prod_{k=J}^{K-1}e^{-i\frac{H_1(k/n)}{n}}e^{-i(J/n - s) H_1((J-1)/n)}
%\end{equation}
%where $K = \ent{nt}$ and $J = \ent{ns}+1$ if $t> s$ and 
%$$
%U_1(n,t,s) = U_1(n,s,t)^*
%$$
%if $t< s$ and $U_1(n,t,t) = Id$. This means that in order to compute $U_1(t,s)$ one can cut the interval $[s,t]$ in sub-intervals of size $\frac{1}{n}$ and use the propagator $e^{i\tau H_1((J-1)/n)}$ on $[s,J/n[$, then $e^{i\tau H_1(J/n)}$ between $J/n$ and $(J+1)/n$ and so on, and then pass to the limit for $n$ going to $\infty$. 
%\end{remark}

\begin{proof}[Proof of Proposition \ref{lem-cont}.]
In order to prove this result, we need to give a closer look to the construction of the propagator $U_1(t,s)$. According to \cite{Kato53}, it appears that $U_1(t,s)$ is the limit as an operator from $ H^1$ to $L^2$ of
\begin{equation}\label{U1seq}
U_1(n,t,s) = e^{-i(t-K/n)H_1(K/n)}\prod_{k=J}^{K-1}e^{-i\frac{H_1(k/n)}{n}}e^{-i(J/n - s) H_1((J-1)/n)}
\end{equation}
where $K = \ent{nt}$ and $J = \ent{ns}+1$ if $t> s$ and 
$$
U_1(n,t,s) = U_1(n,s,t)^*
$$
if $t< s$ and $U_1(n,t,t) = Id$. This means that in order to compute $U_1(t,s)$ one can cut the interval $[s,t]$ in sub-intervals of size $\frac{1}{n}$ and use the propagator $e^{i\tau H_1((J-1)/n)}$ on $[s,J/n[$, then $e^{i\tau H_1(J/n)}$ between $J/n$ and $(J+1)/n$ and so on, and then pass to the limit for $n$ going to $\infty$. 
Therefore, we preliminarily prove the following
\begin{lemma}\label{lemn} 
Let $q$ and $Z$ satisfy \eqref{asszone},\eqref{smallvelone} and \eqref{smallaccone}. Then the sequence $U_1(n,t,s)$ defined by \eqref{U1seq} is uniformly bounded in $n,t,s$ as an operator from $L^2$ to itself and as an operator from $ H^1$ to itself. More precisely,
\begin{eqnarray*}
\|U_1(n,t,s)\|_{L^2\rightarrow L^2} &\leq & 1 \\
\|U_1(n,t,s)\|_{H^1 \rightarrow H^1} & \leq & C_1e^{C_1 \|\ddot q\|_{L^1}} .
\end{eqnarray*}
\end{lemma}
\begin{remark}\label{rkinterpol}
By standard interpolation argument, we can deduce that $U_1(n,t,s)$ is also uniformly bounded as an operator from $H^\sigma$ to itself for any $\sigma\in[0,1]$.
\end{remark}
\begin{proof}[Proof of Lemma \ref{lemn}. ] Thanks to Lemma \ref{lem-pert}, for $m=0,1$
$$
\|e^{itH_1(t_0)}u\|_{ H^m} \leq 2 \|H_1(t_0)^m e^{iH_1(t_0)t} u\|_{L^2}
$$
and since $H_1(t_0)$ and $e^{itH_1(t_0)}$ commute
$$
\|e^{itH_1(t_0)}u\|_{ H^m} \leq 2 \| e^{iH_1(t_0)t} H_1(t_0)^mu\|_{L^2}
$$
and given that $H_1(t_0)$ is essentially self adjoint
$$
\|e^{itH_1(t_0)}u\|_{H^m} \leq 2 \|  H_1(t_0)^mu\|_{L^2}.
$$
Let us now prove that 
\begin{equation}\label{ww}
H_1(\frac{K-1}{n})^m\prod_{k=J}^{K-1}e^{-i\frac{H_1(k/n)}{n}}H_1(\frac{J}{n})^{-m}
\end{equation}
is bounded in $\mathcal L(L^2)$ uniformly in $n$, which would imply the result. We can rewrite \eqref{ww} as
\begin{multline*}
\Big( \prod_{k=J+1}^{K-1}H_1(k/n)^me^{-i\frac{H_1(k/n)}{n}}H_1(\frac{k}{n})^{-m}H_1(\frac{k}{n})^mH_1(\frac{k-1}{n})^{-m}\Big)\\
H_1(\frac{J}{n})^me^{-i\frac{H_1(J/n)}{n}}H_1(\frac{J}{n})^{-m}.
\end{multline*}
Since $H_1(k/n)$ and $e^{-i\frac{H_1(k/n)}{n}}$ commute, this is equal to
$$
\Big( \prod_{k=J+1}^{K-1}e^{-i\frac{H_1(k/n)}{n}}H_1(\frac{k}{n})^mH_1(\frac{k-1}{n})^{-m}\Big)e^{-i\frac{H_1(J/n)}{n}},
$$
and since $H_1$ is essentially self adjoint, we get $\|e^{-i\frac{H_1(k/n)}{n}}\|_{L^2 \rightarrow L^2} =1 $. Thus
$$
\big\| H_1(\frac{K-1}{n})^m\prod_{k=J}^{K-1}e^{-i\frac{H_1(k/n)}{n}}H_1(\frac{J}{n})^{-m}\big\|_{L^2 \rightarrow L^2} \leq \prod_{k=J+1}^K \|H_1(\frac{k}{n})^mH_1(\frac{k-1}{n})^{-m}\|_{L^2\rightarrow L^2}.
$$
For all $k$,
\begin{multline*}
\|H_1(\frac{k}{n})^mH_1(\frac{k-1}{n})^{-m}\|_{L^2\rightarrow L^2} = \|1 + (H_1(\frac{k}{n})^m- H_1(\frac{k-1}{n})^m)H_1(\frac{k-1}{n})^{-m}\|_{L^2\rightarrow L^2} 
\\
\leq 1 + C \|H_1(\frac{k}{n})- H_1(\frac{k-1}{n})\|_{H^1 \rightarrow L^2}.
\end{multline*}
We get
$$
\prod_{k=J+1}^K \|H_1(\frac{k}{n})^mH_1(\frac{k-1}{n})^{-m}\|_{L^2\rightarrow L^2} \leq e^{C\sum_{k=J+1}^K  \|H_1(\frac{k}{n})- H_1(\frac{k-1}{n})\|_{H^1 \rightarrow L^2}},
$$
which is uniformly bounded in $n$ since $H_1$ is of bounded variations. Indeed, 
$$
\sum_{k=J+1}^K  \|H_1(\frac{k}{n})- H_1(\frac{k-1}{n})\|_{H^1 \rightarrow L^2}\leq \sum_{k=J+1}^K  \|(\dot q(\frac{k}{n})-\dot q (\frac{k-1}{n})) \cdot \grad\|_{H^1 \rightarrow L^2}\leq \|\ddot q\|_{L^1}
$$
We deduce from that the result.
\end{proof}

We now get back to the proof of Proposition \ref{lem-cont}. Writing $H_1(q,t) = H_1(t)$ and $U_1(n,q,t,s) = U_1(n,t,s)$, we have 
\begin{multline*}
\|U_1(n,q_1,t,s) - U_1(n,q_2,t,s)\|_{H^{\sigma} \rightarrow H^{\sigma-1}}\\ \leq 
\|e^{-i(t-\frac{K}{n})H_1(q_1,\frac{K}{n})} - e^{-i(t-\frac{K}{n}) H_1(q_2,\frac{K}{n})}\|_{H^{\sigma} \rightarrow H^{\sigma-1}} \|U_1(n,q_1,\frac{K}{n},s)\|_{H^{\sigma} \rightarrow H^{\sigma}} +\\
\| e^{-i(t-\frac{K}{n}) H_1(q_2,t)}\|_{H^{\sigma-1} \rightarrow H^{\sigma-1}} \|U_1(n,q_1,\frac{K}{n},s) - U_1(n,q_2,\frac{K}{n},s)\|_{H^{\sigma} \rightarrow H^{{\sigma-1}}}.
\end{multline*}\normalsize

Because of the boundedness from $H^{\sigma-1}$ to $H^{\sigma-1}$ of $e^{itH_1(t_0)}$ and from $H^{\sigma}$ to $H^{\sigma}$ of $U_1(n,q_1,\frac{K}{n},s)$ for all $t,t_0$ (see Remark \ref{rkinterpol}) and Proposition \ref{prop-linH1}, we get that
$$
\|e^{-i(t-K/n)H_1(q_1,s+K/n)} - e^{-i(t-K/n) H_1(q_2,t)}\|_{H^{\sigma} \rightarrow H^{\sigma-1}}
$$
$$
 \leq 3 |t-K/n| \|H_1(q_1,s+K/n) - H_1(q_2,s+K/n)\|_{H^{\sigma} \rightarrow H^{\sigma-1}}.
$$
This ensures that
\begin{multline*}
\|U_1(n,q_1,t,s) - U_1(n,q_2,t,s)\|_{H^{\sigma} \rightarrow H^{\sigma-1}} \\\leq Ce^{C \|\ddot q\|_{L^1}} |t-K/n| \|\dot{q}_1 - \dot{q}_2\|_{L^\infty} +
\|U_1(n,q_1,K/n,s) - U_1(n,q_2,K/n,s)\|_{H^{\sigma} \rightarrow H^{\sigma-1}}
\end{multline*}
which yields by induction
$$
\|U_1(n,q_1,t,s) - U_1(n,q_2,t,s)\|_{H^{\sigma} \rightarrow H^{\sigma-1}} \leq Ce^{C \|\ddot q\|_{L^1}} |t-s| \|\dot{q}_1 - \dot{q}_2\|_{L^\infty} .
$$
By letting $n$ go to $\infty$, we get
\begin{equation}\label{H1cont}
\|U_1(q_1,t,s) - U_1(q_2,t,s)\|_{H^{\sigma} \rightarrow H^{\sigma-1}} \leq Ce^{C \|\ddot q\|_{L^1}} |t-s| \|\dot{q}_1 - \dot{q}_2\|_{L^\infty}  .
\end{equation}
What is more, we have  $U_{q}(t,s) = I_q(t)^{-1} U_1(q,t,s) I_q(s)$ and $I(t)$ is an isometry of $H^{\sigma}$, and satisfies
$$
\|I_{q_1}(t) - I_{q_2}(t)\|_{H^{{\sigma}}\rightarrow H^{\sigma-1}} \lesssim |q_1(t) - q_2(t)| \leq T \|\dot{q}_1 - \dot{q}_2\|_{L^\infty}.
$$
Indeed, let $u\in H^{\sigma}$, $v\in H^{1-\sigma}$ and define $F_q := \int \overline v I_q(t) u$.\\
 We have $\grad_q F = \int \overline v I_q(t) \grad u$, from which we get $|\grad_q F| \leq \|v\|_{H^{1-\sigma}}\|u\|_{H^{\sigma}}$. And thus
$$
|\an{v, I_{q_1}(t)u - I_{q_2}(t)u}| \leq |F_{q_1} -  F_{q_2}| \leq |q_1(t)   - q_2(t)| \; \|v\|_{H^{1-\sigma}}\|u\|_{H^{\sigma}}.
$$

We get
$$
\|I_{q_1}(t) - I_{q_2}(t)\|_{H^{\sigma}\rightarrow H^{\sigma-1}}  \lesssim |q_1(t) - q_2(t)| \leq T \|\dot{q}_1 - \dot{q}_2\|_{L^\infty}.
$$

We finally get
$$
\|U_{q_1}(t,s) - U_{q_2}(t,s)\|_{H^{\sigma} \rightarrow H^{\sigma-1}} \leq Ce^{C \|\ddot q\|_{L^1}} T \|\dot{q}_1 - \dot{q}_2\|_{L^\infty}   .
$$

\end{proof}

Therefore, putting together Proposition \ref{prop-linH1} and \ref{lem-cont} we obtain the proof of Theorem \ref{teo1} in the case of a single nucleus.

\subsection{Several nuclei}\label{multi}

We now present the necessary modifications to deal with the case of several moving nuclei, i.e. when the potential $V$ takes the form
$$
V(x,t)=-\sum_{k=1}^N\frac{Z_k}{|x-q_k(t)|},\qquad q_k(0)=a_k,
$$
where the $a_k$ satisfy
$$
\min \{|a_k - a_l|\ |\  k\neq l\} = 8 \varepsilon_0 > 0.
$$
To ensure that at every time $t\in [0,T]$ we have 
$$
\min \{|q_k(t) - q_l(t)|\  |\  k\neq l\} = 4 \varepsilon_0 > 0,
$$
we require for all $k$ that 
$$T\sup_{t} |\dot q_k(t)| \leq 2 \varepsilon_0.$$

The idea is to adapt the strategy developed for a single nucleus by introducing suitable cutoffs in order to deal with each singularity simultaneously and without mutual interference. This means that instead of doing the change of variable, $x\leftarrow x+q(t)$ as in the single nucleus case, we do the change $ x\leftarrow x+ q_k(t) - a_k$ but only around the singularity $q_k(t)$. Since $q_k(t)$ and $a_k$ are close (at least for small times), we can do it only around $a_k$. We are inspired by a similar strategy used by Kato and Yajima in \cite{katoyajim}, who define a "local pseudo-Lorentz" transformation mapping retarded Lienard-Wiechert potentials to Coulomb potentials at fixed positions.
To define this change of variable, we introduce a symmetric, real valued cut-off function $\zeta(x)=\zeta(|x|)$ (with a slight abuse of notation we will denote with $\zeta$ both the function and its radial component) having the following properties: 
\medskip
\begin{itemize}
\item
$\zeta(|x|)\in C^\infty (\mathbb{R}^+)$;
\medskip
\item
$\zeta(|x|)=1$ for $|x|\leq 1$;
\medskip
\item
$\zeta(|x|)=0$ for $|x|\geq 2$;
\medskip
\item
$\zeta(|x|)\in[0,1]$;
\medskip
\item
$\zeta'(|x|)|\leq3/2$.
\end{itemize}

In view of constructing our simultaneous "nuclei-freezing" transformation, we introduce for each $k=1,\dots N$ the functions
\begin{equation}\label{transf}
\phi^k(t,x)=x+\Tk(x,t)
\end{equation}
where we are denoting with 
$$
\Tk(x,t)=\zeta\left(\frac{x-a_k}{\varepsilon_0}\right)(q_k(t)-a_k).
$$
We write 
$$
\phi(t,x) = x + \sum_k \Tk(x,t)
$$
and
$$
 \Phi(t) : u\mapsto u(t,\phi(t,x)).
$$

Note that here $\Phi(t)$ replaces the $I(t)$ defined by \eqref{It}. We also remark that from now on, constants may not only depend on $(Z_k)_k$ but also on $\varepsilon_0$ but in any case, they do not depend on $T$.

\begin{lemma}\label{lem-stabL2} There exist constants $C$ and $M_0$ such that for all $t\in [0,T]$, $T \sup_k\|\dot q_k\|_{L^\infty}\leq M_0$ and all $u \in  L^2$,
$$
\frac1{C} \|u\|_{L^2} \leq \|\Phi(t) u\|_{L^2} \leq C \|u\|_{L^2}.
$$
\end{lemma}

\begin{proof} By performing the change of variable $x\leftarrow \phi(t,x)$, we get
$$
\|\Phi(t) u\|_{L^2} = \|\textrm{jac}(\phi(t))^{-1/2} u\|_{L^2},
$$
where $\textrm{jac}(\phi(t)) = |\textrm{det} \textrm{Jac}(\phi(t))|$ and $\textrm{Jac}(\phi(t))$ is the Jacobian matrix associated to $\phi$. 

Given that
$$
\partial_j \phi(t,x) = e_j + \sum_k \zeta'\Big( \Big| \frac{x-a_k}{\varepsilon_0}\Big|\Big) \frac{x_j - a_{k,j}}{ \varepsilon_0 |x - a_k|} (q_k(t) - a_k)
$$
where $e_j$ is the $j$-th vector of the canonical basis of $\R^3$, and $a_{k,j}$ is the $j$-th coordinate of $a_k$, and given that the supports of $x\mapsto  \zeta'\Big( \Big| \frac{x-a_k}{\varepsilon_0}\Big|\Big)$ are disjoint and away from $a_k$, we get that
$$
|\partial_j \phi(t,x) - e_j| \leq \frac32 T \sup_k\|\dot q_k\|_{L^\infty}\frac1{\varepsilon_0}
$$
and thus that
$$
|\textrm{Jac } (\phi(t)) - I_3| \lesssim T \sup_k\|\dot q_k\|_{L^\infty}.
$$
Therefore there exists $M_0$ such that for all $T \sup_k\|\dot q_k\|_{L^\infty}\leq M_0$, $x\mapsto \phi(t,x)$ is a bijection and such that for all $t\in [0,T]$,
$$
|\textrm{jac }(\phi(t))^{1/2}  - 1| \lesssim T \sup_k\|\dot q_k\|_{L^\infty} \textrm{ and }|\textrm{jac }(\phi(t))^{-1/2}  - 1|\lesssim T \sup_k\|\dot q_k\|_{L^\infty}
$$
which ensures the lemma.
\end{proof}

\begin{lemma} \label{lem-stabH1} There exist $M_0$ and $C$ such that for all $T \sup_k\|\dot q_k\|_{L^\infty}\leq M_0$ all $u \in H^1$ and all $t\in [0,T]$, the following representations and estimates hold true
\begin{eqnarray*}
\nabla (\Phi(t) u) = (\nabla u) \circ \phi(t) + P_t(u) & \ \ \  \textrm{ where }\ \ \  \|P_t\|_{H^1 \rightarrow L^2} \lesssim T \sup_k\|\dot q_k\|_{L^\infty} \\
\nabla (\Phi(t)^{-1} u) = (\nabla u)\circ \phi(t)^{-1} + Q_t(u) & \ \ \ \textrm{ where }\ \ \ \|Q_t\|_{H^1\rightarrow L^2} \lesssim T \sup_k\|\dot q_k\|_{L^\infty} 
\end{eqnarray*}
and
$$
\frac1{C} \|u\|_{H^1} \leq \|\Phi(t) u \|_{H^1} \leq C \|u\|_{H^1} \; ,\; \frac1{C} \|u\|_{H^2} \leq \|\Phi(t) u \|_{H^2} \leq C \|u\|_{H^2}.
$$

\end{lemma}

\begin{proof} We have 
$$
\nabla (\Phi(t) u) = \textrm{ Jac }(\phi(t))(\nabla u) \circ \phi(t).
$$
As $\textrm{Jac}(\phi(t))$ is a perturbation of the identity, we get
$$
\nabla (\Phi(t) u) = (\nabla u)\circ \phi(t) + A(t,x) (\nabla u) \circ \phi(t),
$$
where $A(t,x)$ is a matrix whose norm is uniformly bounded by $T \sup_k\|\dot q_k\|_{L^\infty}$. We have 
$$
\| A(t,x) (\nabla u) \circ \phi\|_{L^2} \lesssim T \sup_k\|\dot q_k\|_{L^\infty} \|(\nabla u) \circ \phi\|_{L^2}
$$
and thanks to Lemma \ref{lem-stabL2}
$$
\| A(t,x) (\nabla u) \circ \phi\|_{L^2} \lesssim T \sup_k\|\dot q_k\|_{L^\infty} \|u\|_{H^1}.
$$
Hence, 
\begin{equation}\label{Ft}
P_t(u) = A(t,x) (\nabla u) \circ \phi(t)
\end{equation}
 satisfies the stated properties.

For the same reasons, we get
$$
Q_t(u) = B(t,x) (\grad u) \circ \phi(t)^{-1}
$$
with $B(t,x) = \textrm{Jac }(\phi(t)^{-1})-I_3$.

Since $\textrm{Jac }(\phi(t))$ is close to the identity, so is its inverse $\textrm{Jac }(\phi(t)^{-1})$ and thus $\|B(t,x)\|\lesssim T \sup_k\|\dot q_k\|_{L^\infty}$. This yields
$$
\|Q_t(u)\|_{L^2} \lesssim T \sup_k\|\dot q_k\|_{L^\infty} \|\nabla u\circ \phi(t)^{-1}\|_{L^2}
$$
and thanks to Lemma \ref{lem-stabL2}, we get that 
$$
\|Q_t(u)\|_{L^2} \lesssim T \sup_k\|\dot q_k\|_{L^\infty} \|u\|_{H^1}.
$$

The equalities involving $\nabla (\Phi(t) u)$ and $\nabla (\Phi(t)^{-1}u)$ ensure the validity of inequalities for $\|\Phi(t) u\|_{H^1}$.

For $H^2$, we have to compute $\lap (\Phi(t) u)$. We have
$$
\lap (\Phi(t) u) = (\Jac (\phi(t) ) \nabla) \cdot \nabla u \circ \phi(t) + \nabla . P_t(u).
$$
As $\Jac (\phi(t))$ is a uniform perturbation of the identity, we have that \\$(\Jac (\phi(t)) \nabla) \cdot \nabla u \circ \phi(t)$ is a uniform perturbation of $(\lap u)\circ \phi(t)$. Moreover, we have
$$
\nabla \cdot P_t(u) = \Jac (\phi(t)) \Big ((\nabla \cdot A(t,x) ) \cdot \nabla u \circ \phi(t) + (A(t,x) \nabla)\cdot \nabla u\circ \phi(t) \Big) .
$$
We have 
$$
\|(A(t,x) \nabla)\cdot \nabla u  \|_{L^2} \lesssim T \sup_k\|\dot q_k\|_{L^\infty} \|u\|_{H^2}.
$$
Besides, $\nabla \cdot A(t,x) = \lap \phi(t)$, which gives
$$
\nabla \cdot A(t,x) = \sum _k \Big[ \frac2{|x-a_k|\varepsilon_0}\zeta' \left(\frac{|x-a_k|}{\varepsilon_0}\right) + \frac1{\varepsilon_0^2} \zeta''\left(\frac{|x-a_k|}{\varepsilon_0}\right)\Big] (q_k(t)-a_k).
$$
Assuming $\varepsilon_0 < 1$ and using that the supports of $\zeta'\left(\frac{|x-a_k|}{\varepsilon_0}\right))$ are outside a ball of center $a_k$ and radius $\varepsilon_0$, we get
$$
\|(\nabla \cdot A(t,x) ) \cdot \nabla u\|_{L^2} \lesssim T \sup_k\|\dot q_k\|_{L^\infty} \|u\|_{H^2}.
$$
(We recall that the supports of $\zeta \Big( \frac{|x-a_k|}{\varepsilon_0}\Big)$ are disjoint.)

This ensures that 
$$
\|\Phi(t) u\|_{H^2} \leq C \|u\|_{H^2}
$$
if $T \sup_k\|\dot q_k\|_{L^\infty}$ is small enough. The reverse inequality can be deduced in the same way.

\end{proof}

The following result is the analogue of Lemma \ref{lem-changevar} in the multinuclear case.
\begin{lemma}\label{prop-changevar+} If $u$ satisfies 
$
i\partial_t u = H(t) u
$
then $v = \Phi(t) u$ satisfies 
$
i\partial_t v = H_N(t) v
$
with 
\begin{equation}\label{multiHn}
H_N(t)v = \D v+\beta v - \sum_k \frac{Z_k}{|x-a_k|} v
\end{equation}
$$+ i\partial_t \phi(t,x)\cdot  \Big( \nabla v - P_t(\Phi(t)^{-1}v)\Big) +\overrightarrow{\alpha}\cdot P_t(\Phi(t)^{-1}v)+R(t,x) v  
$$
where $P_t$ is defined as in Lemma \ref{lem-stabH1} (see also \eqref{Ft}),
$$
R(t,x) = \sum_k Z_k \Big( \frac1{|x-a_k|} - \frac1{|\phi(t,x) - q_k(t)|}\Big).
$$
and $\overrightarrow{\alpha}$ is the vector of the $\alpha_k$ matrices.

\end{lemma}

\begin{proof} Straightforward computation. \end{proof}

In view of defining the two-parameter propagator and use Kato's theory as in the one nucleus case, we need the following

\begin{proposition}\label{prop-estimH1sn} Assume that for all $k$, $|Z_k| < \frac{\sqrt 3}{2}$. There exist $C_1$, and $C$ such that if the trajectories $q_k(t)$ are such that 
\begin{equation}\label{hpqnuclei}
(1+T)\sup_{k}\|\dot q_k(t)\|_{L^\infty}\leq C_1,\quad \sup_{k} \|\ddot{q}_k\|_{L^1([0,T])} < \infty,
\end{equation}
then there exists $\delta\in \R_+$ such that for all $t\in [0,T]$,
$$
-\delta\|u\|_{L^2}^2+ \frac1{C}\|u\|_{H^1}^2 \leq \|H_N(t)u\|_{L^2}^2 \leq C \|u\|_{H^1}^2
$$
and
\begin{equation}\label{crucbound}
\| i\partial_t H_N(t)\|_{L^1([0,T],H^1 \rightarrow L^2)} \leq C\sup_k( \|\ddot q_k\|_{L^1} + T \|\dot q_k\|_{L^\infty}).
\end{equation}
Moreover, for any $\sigma\in[1,\frac32)$, the operator $H_N(t)$ is bounded from $H^{\sigma}$ into $H^{\sigma-1}$.

\end{proposition}

\begin{proof} Thanks to the usual theory of essential self-adjointness of Dirac operators with Coulomb potentials, the proposition is already true for 
$$
\D v+\beta v - \sum_k \frac{Z_k}{|x-a_k|} v.
$$

We estimate the other terms. We have 
$$
i\partial_t \phi(t,x) = \sum_k \zeta\Big( \frac{|x-a_k|}{\varepsilon_0}\Big) \dot{q}_k(t),
$$
hence 
$$
\|i\partial_t \phi(t,x)\cdot  ( \nabla v - P_t(\Phi(t)^{-1}v))\|_{L^2} \lesssim \sum_k \sup_{t}|\dot q_k(t)| \|v\|_{H^1}.
$$
We have 
$$
 \|\overrightarrow{\alpha}\cdot P_t(\Phi(t)^{-1}v)\|_{L^2} \lesssim T \sup_k\|\dot q_k\|_{L^\infty} \|v\|_{H^1}.
$$
Finally, since $\frac1{|x-a_k|} - \frac1{|\phi(t,x) - q_k(t)|}$ is supported outside the ball of center $a_k$ and radius $\varepsilon_0$, and given that outside this ball, $|\phi(t,x) - q_k(t)| \geq |x-a_k| - |q_k(t) - a_k|$ we have that for $T \sup_k\|\dot q_k\|_{L^\infty}\leq \frac12$, $|\phi(t,x) - q_k(t)|\geq \frac{\varepsilon_0}{2}$. Therefore
$$
\|R(t,x)\|_{L^\infty([0,T]\times \R^3)} \lesssim \varepsilon_0^{-1} \sum_k |Z_k|.
$$

Therefore, assuming that the quantity $(1+T) \sup_k\|\dot q_k\|_{L^\infty}$ is small enough, we get that $H_N(t)$ satisfies
$$
-\delta\|u\|_{L^2}^2+ \frac1{C}\|u\|_{H^1}^2 \leq \|H_N(t)u\|_{L^2}^2 \leq C \|u\|_{H^1}^2 ;
$$
the $L^2$ norm appears in the term involving $R(t,x)$.

Computing the derivative of $H_N(t)$ yields
$$
\partial_t H_N(t) = i\partial_t^2 \phi(t) \cdot \Big( \nabla - P_t\circ \Phi(t)^{-1}\Big) - i\partial_t \phi \cdot \partial_t P_t \circ \Phi(t)^{-1}  + i\overrightarrow{\alpha} \cdot \partial_t P_t \circ \Phi(t) + \partial_t R(t,x).
$$
We have $\|\partial_t^2 \phi\|_{L^1} \lesssim \sup_{k} \|\ddot{q}_k\|_{L^1}$  and $\|\nabla - P_t\circ \Phi(t)^{-1}\|_{L^\infty([0,T],H^1\rightarrow L^2)}\leq C$ hence 
$$ 
\|i\partial_t^2 \phi(t) \cdot \Big( \nabla - P_t\circ \Phi(t)^{-1}\Big)\|_{L^1([0,T],H^1\rightarrow L^2)} \lesssim \sup_{k} \|\ddot{q}_k\|_{L^1}.
$$
We have $|i\partial_t \phi|\lesssim \sup_{k,t} |\dot{q}_k(t)|$ and $P_t\circ \Phi(t)^{-1} = A(t,x) \Jac (\Phi(t))^{-1} \nabla$. What is more,
$$
T|\partial_t A| = T|\partial_t \Jac (\phi(t))|   \lesssim T \sup_k\|\dot q_k\|_{L^\infty}
$$ 
and $\partial_t \Jac (\Phi(t))^{-1} = -\Jac (\Phi(t))^{-1} \partial_t A \Jac (\Phi(t))^{-1}$, hence
$$
\|\partial_t \phi\|_{L^2\rightarrow L^2} \lesssim C_1 \textrm{ and }T\|\partial_t P_t\circ \Phi(t)^{-1}\|_{H^1\rightarrow L^2} \lesssim T \sup_k\|\dot q_k\|_{L^\infty}.
$$
This gives
$$
T\sup_t\|i\partial_t \phi \cdot \partial_t P_t \circ \Phi(t)^{-1}  \|_{H^1\rightarrow L^2} \lesssim T \sup_k\|\dot q_k\|_{L^\infty} C_1
$$
and 
$$
T \sup_t \| i\overrightarrow{\alpha} \cdot \partial_t P_t \circ \Phi(t)\|_{H^1\rightarrow L^2 } \lesssim T \sup_k\|\dot q_k\|_{L^\infty}.
$$
Finally, as 
$$\partial_t \Big( \frac1{|x-a_k|} - \frac1{|\phi(t,x) - q_k(t)|}\Big) = \frac{(\partial_t \phi - \dot q_k(t))(\phi - q_k)}{|\phi-q_k|^3}$$ and it is supported outside the ball of center $a_k$ and radius $\varepsilon_0$, we get
$$
T\sup_t \| \partial_t R(t,x) \|_{H^1\rightarrow L^2 } \lesssim  T \sup_k\|\dot q_k\|_{L^\infty},
$$
which yields the result as $\|\cdot\|_{L^1([0,T]}\leq T \|\cdot\|_{L^ \infty([0,T])}$.
The last statement of the Theorem is proved when the remainders with respect to the multicenter Coulomb operator appearing in \ref{multiHn} are controlled. Multiplications by smooth decaying function are bounded operators in Sobolev spaces (see \cite{bah}, in particular Thm 1.62) and moreover the term $R$ can be bounded as in the proof of \ref{lem-stabL2}. This gives the result.
\end{proof}

We are now in position of proving the main result of this subsection.

\begin{proposition}\label{moreregprop2}
 Let $Z_k$ be such that $|Z_k| < \frac{\sqrt 3}{2}$ for each $k$, and let $q_k(t)$ satisfy assumptions \eqref{hpqnuclei}. Then the flow of the equation
$$
i\partial_t u = H(t) u
$$
with $H(t)$ given by \eqref{ndimham} is given by a family of operators $U(t,s)=U_q(t,s)$ satisfying 
$$
i\partial_t U_q(t,s) = H(t)U_q(t,s) \; , \; i\partial_s U_q(t,s) = - U_q(t,s)H(s)
$$
and 
$$
U_q(t,s) \circ U_q(s,r) = U_q(t,r).
$$ 
with 
$$U_q \in \mathcal C([0,T]^2,\mathcal L(H^\sigma))$$
for any $\sigma\in[0,\frac32)$. In particular the norms 
$$
\|U_q(t,s)\|_{H^\sigma\rightarrow H^\sigma}
$$
are uniformly bounded in $t$, $s$, and $q$.\end{proposition}
\begin{proof}
As in the single nucleus case, this result is a consequence of Kato's theory and Proposition \ref{prop-estimH1sn}. We omit the details.

\end{proof}

The following result is an ingredient needed for the continuity of the propagator $U_q$, that is forthcoming Proposition \ref{cor-contUsev}.

\begin{proposition}\label{prop-qdep} Let $q^{(1)} = (q_1^{(1)},\hdots ,q_N^{(1)})$ and $q^{(2)} = (q_1^{(2)},\hdots ,q_N^{(2)})$ be two vectors of $\mathcal C^2([0,T])$ satisfying assumptions \eqref{hpqnuclei}. Let $H_{N,j}(t)$ (resp $\Phi_j(t)$) be the operator $H_N(t)$ (resp. $\Phi(t)$) associated to $q^{(j)}$. We assume that $q^{(1)}(0) = q^{(2)}(0) = (a_1,\hdots ,a_N)$. Then for all $t\in [0,T]$ and all $\sigma\in[1,\frac32)$,
$$
\|H_{N,1}(t) - H_{N,2}(t)\|_{H^{\sigma} \rightarrow H^{\sigma-1}} \leq C  \sup_{k} (1+T) \|(\dot{q_k}^{(1)})- (\dot{q_k}^{(2)})\|_{L^\infty}.
$$
\end{proposition}

\begin{proof} We have
$$
\partial_j \phi_1(t,x) - \partial_j \phi_2 (t,x) = \sum_k \zeta'\Big( \Big| \frac{x-a_k}{\varepsilon_0}\Big|\Big) \frac{x_j - a_{k,j}}{ \varepsilon_0 |x - a_k|} (q_k^{(1)}(t) - q_k^{(2)}(t)).
$$
Hence
$$
|\Jac(\phi_1 ) - \Jac (\phi_2)| \leq \sup_{k,t} T |(\dot{q_k}^{(1)})- (\dot{q_k}^{(2)})|.
$$
This gives in particular
$$
\|P_{t,1}(\Phi_1(t)^{-1}v))- P_{t,2}(\Phi_2(t)^{-1}v\|_{L^2} \lesssim \sup_{k,t} T |(\dot{q_k}^{(1)})- (\dot{q_k}^{(2)})| \|\nabla v\|_{L^2}
$$
hence
$$
\|P_{t,1}(\Phi_1(t)^{-1}))- P_{t,2}(\Phi_2(t)^{-1}\|_{H^\sigma\rightarrow H^{\sigma-1}} \lesssim \sup_{k,t} T |(\dot{q_k}^{(1)})- (\dot{q_k}^{(2)})|.
$$

We also have 
$$
|i\partial_t \phi_1 - i\partial_t \phi_2| \leq \sup_{k,t} |(\dot{q_k}^{(1)})(t) - (\dot{q_k}^{(2)})(t)|.
$$

And finally,
$$
R_1(t,x) - R_2(t,x) = \sum_k Z_k \Big( \frac1{|\phi(t,x) - q_k^{(2)}(t)|} - \frac1{|\phi(t,x) - q_k^{(1)}(t)|}\Big)
$$
which yields 
$$
R_1(t,x) \leq C_Z  \sup_{k,t} T |(\dot{q_k}^{(1)})- (\dot{q_k}^{(2)})|.
$$

This concludes the proof.\end{proof}

To conclude with, we prove the continuity of the propagator $U_q$ with respect to $q$.

\begin{proposition}\label{cor-contUsev} Let $q^{(1)} = (q_1^{(1)},\hdots ,q_N^{(1)})$ and $q^{(2)} = (q_1^{(2)},\hdots ,q_N^{(2)})$ be two vectors of $\mathcal C^2([0,T])$ satisfying assumptions \eqref{hpqnuclei}.  
We assume that $q^{(1)}(0) = q^{(2)}(0) = (a_1,\hdots ,a_N)$. There exists $C$ (independent from $T$ and $q^{(1)}, q^{(2)}$) such that for all $t,s \in [0,T]^2$, we have for any $\sigma\in [1,\frac 32)$,
$$
\|U_{q^{(1)}}- U_{q^{(2)}}\|_{H^{\sigma}\rightarrow H^{\sigma-1}} \leq  C \sup_{k}(1+ T) \|(\dot{q_k}^{(1)})- (\dot{q_k}^{(2)})\|_{L^\infty}
$$
\end{proposition}

\begin{proof} What we have to add with regard to the single nucleus case (i.e. Lemma \ref{lem-cont}), is that we also have

$$
\|\Phi_1(t) - \Phi_2(t)\|_{H^{\sigma}\rightarrow H^{\sigma-1}} \lesssim \sup_{k} T \|(\dot{q_k}^{(1)})- (\dot{q_k}^{(2)})\|_{L^\infty}.
$$ 
such that we can bound the difference between the operators due to the different changes of variables.
\end{proof}

Again, putting together Propositions \ref{moreregprop2} and \ref{cor-contUsev} we obtain Theorem \ref{teo1} in the multi-nuclei case, so its proof is concluded.

\section{Local well-posedness of the electron-nuclei dynamics}\label{proof}

This section is devoted to the proof of our main results, Theorem \ref{teo2},\ref{teo3}. In all this section, as stated in the hypotheses of the cited Theorems, we assume that the $q_k$ satisfy the separation assumption and that $|Z_k| < \frac{\sqrt 3}{2}$ for all $k$. In all the results, the constants may depend on $Z_k$ and $\varepsilon_0$. 
 
\subsection{Nonlinear estimates}

In this subsection we collect some preliminary estimates that will be needed in the sequel; we will include some proofs for the sake of completeness. We start by recalling this classical version of generalized Hardy inequalities (see e.g. Theorem 2.57 in \cite{bah})

%\textcolor{red}{WE WILL NEED TO MOVE THIS THING AHEAD
%\begin{theorem}\label{genhardy}
%For any $\sigma\in[0,\frac{d}2)$ there exists a constant $C$ such that for any $f\in \dot{H}^s(\R^d)$ 
%\begin{equation}\label{hardgen}
%\int_{\R^d}\frac{|f(x)|^2}{|x|^{2s}}dx\leq C\|f\|_{\dot{H}^s}.
%\end{equation}
%\end{theorem}
%}
We now provide some standard estimates for the convolution term.

\begin{lemma}\label{lem-multiL2} Let $u,v,w \in  H^1$. Then the following estimates hold
$$
\|(uv * |x|^{-1})w\|_{L^2} \lesssim \|u\|_{L^2}\|v\|_{ H^1}\|w\|_{L^2},
$$
$$
\|(uv * |x|^{-1})w\|_{ H^1} \lesssim \|u\|_{ H^1}\|v\|_{ H^1}\|w\|_{H^1} .
$$
%If moreover $u,v,w$ are in $H^2$ we have 
%$$
%\|(uv * |x|^{-1})w\|_{H^2} \lesssim \|u\|_{H^1}\|v\|_{ H^1}\|w\|_{H^2} + \|w\|_{H^1}(\|u\|_{\dot H^1}\|v\|_{H^2} + \|u\|_{\dot H^2}\|v\|_{H^1})
% $$
\end{lemma}

\begin{proof} The proof of the first inequality is a combination of H\"older's and Hardy's inequalities. Indeed,
$$
\|(uv * |x|^{-1})w\|_{L^2} \leq \|(uv * |x|^{-1})\|_{L^\infty} \|w\|_{L^2}
$$
and for all $x$
$$
\Big| \int (uv)(x-y) |y|^{-1}dy \Big| \leq \|u_x\|_{L^2} \|v_x|y|^{-1}\|_{L^2}
$$
where $u_x(y) = u(x-y)$. By a change of variable, $\|u_x\|_{L^2} = \|u\|_{L^2}$ and by Hardy's inequality $\|v_x|y|^{-1}\|_{L^2} \leq 4 \|\nabla v_x\|_{L^2}$. Given that $\nabla v_x = -(\nabla v)_x$, we get the result.

For the second inequality, we write 
$$
\nabla ((uv * |x|^{-1})w) = ((\nabla u) v * |x|^{-1})w + ((u\nabla v)*|x|^{-1}) w + ((uv)*|x|^{-1})\nabla w.
$$
By using the first inequality, we can estimate the $L^2$ norm of the right hand side as follows
\begin{eqnarray*}
\| ((\nabla u) v * |x|^{-1})w \|_{L^2} &\leq &  \|\nabla u\|_{L^2} \|v\|_{\dot H^1} \|w\|_{L^2} \\
\|((u\nabla v)*|x|^{-1}) w\|_{L^2} & \leq & \|u\|_{\dot H^1} \|\nabla v\|_{L^2}\|w\|_{L^2} \\
\|((uv)*|x|^{-1})\nabla w \|_{L^2} & \leq &  \|u\|_{L^2}\|v\|_{\dot H^1}\|\nabla w\|_{L^2}
\end{eqnarray*}
which give the first estimate. 

The estimate in the $H^2$ case can be obtained in the same way as above by using the Laplacian instead of the gradient.

\end{proof}

In what follows we will also need the following fractional versions of the estimates above.
\begin{lemma}\label{lem-multiL2bis}
Let $s\in(0,1/2)$ and $u$, $v$, $w\in H^{s+1}$. Then the following estimate holds
$$
\| (uv * |x|^{-1})w\|_{H^{s+1}}\leq\| u\|_{H^{s+1}}\| v\|_{H^{s+1}}\| w\|_{H^{s+1}}.
$$
\end{lemma}

%\begin{remark}
%We point out that the lack of the case $s=1/2$ in the statement above is due to the strategy of our explicit proof, which essentially relies on Hardy's inequalities \eqref{genhardy}; we expect anyway that such a gap could be filled by relying on multilinear interpolation argument. Anyway, our statement is enough for our scopes.
%\end{remark}

\begin{remark}
As it will be clear from the proof, these estimates are not sharp, but we prefer to present them in this clear way as they will be enough for the scope of this paper.
\end{remark}

\begin{proof}
We start by the case $s\in(0,1/2)$; we deal with the homogeneous Sobolev norm, as this is the problematic term. We write
$$
\| (uv * |x|^{-1})w\|_{\dot{H}^{s+1}}=\| (\nabla(uv) * |x|^{-1})w\|_{\dot{H}^{s}} + \| (uv)*|x|^{-1} \grad w \|_{\dot H^s} .
$$
For $\| (uv)*|x|^{-1} \grad w \|_{\dot H^s}$, we use Leibniz's inequalities to get
$$
\| (uv)*|x|^{-1} \grad w \|_{\dot H^s} \lesssim \|(uv)*|x|^{-1} \|_{ W^{s,\infty}} \|\grad w\|_{H^s} ,
$$
Since $\|\grad w\|_{H^s} \leq \|w\|_{H^{s+1}}$ and $\|(uv)*|x|^{-1} \|_{ W^{s,\infty}} \leq \|(uv)*|x|^{-1} \|_{ W^{1,\infty}}$ and since we have already proved
$$
\|(uv)*|x|^{-1} \|_{ W^{1,\infty}} \lesssim \|u\|_{H^1}\|v\|_{H^1}
$$
in the previous lemma, we get
$$
\| (uv)*|x|^{-1} \grad w \|_{\dot H^s} \lesssim  \|u\|_{H^{s+1}}\|v\|_{H^{s+1}}\|w\|_{H^{s+1}}.
$$

For $\| (\nabla(uv) * |x|^{-1})w\|_{\dot{H}^{s}}$, we deal with one of the two terms appearing after using Leibniz rule on the gradient, the other one being analogous: we have, setting $G=(\nabla u)v * |x|^{-1}$
$$
\| Gw\|_{\dot{H}^{s}}^2=\int_{\R^3\times\R^3}\frac{|G(x)w(x)-G(y)w(y)|^2}{|x-y|^{3+2s}}dxdy\leq
$$
$$
\int_{\R^3\times\R^3}\frac{|G(x)-G(y)|^2|w(x)|^2}{|x-y|^{3+2s}}dxdy+
\int_{\R^3\times\R^3}\frac{|G(y)|^2|w(x)-w(y)|^2}{|x-y|^{3+2s}}dxdy=I+II.
$$
The term $II$ can be directly dealt with by estimating as
$$
II\lesssim \|G\|_{L^\infty}^2\|w\|_{\dot{H}^{s+1}}^2
$$ 
and $\|G\|_{L^\infty}$ has been dealt with in the previous Lemma. For the term $I$ we write instead
$$
I=\int_{\R^3}dx\int_{|x-y|\geq1}\frac{|G(x)-G(y)|^2|w(x)|^2}{|x-y|^{3+2s}}dy+
\int_{\R^3}dx\int_{|x-y|<1}\frac{|G(x)-G(y)|^2|w(x)|^2}{|x-y|^{3+2s}}dy=I_1+I_2.
$$
For the term $I_1$ we get
$$
I_1\leq\int_{\R^3}dx|w(x)|^2\int_{|x-y|\geq1}\frac{2\|G\|_{L^\infty}^2}{|x-y|^{3+2s}}\lesssim \|G\|_{L^\infty}^2\|w\|_{L^2}^2.
$$
as, since $3+2s>3$, which is the dimension, the term $|x-y|^{-3-2s}$ is integrable in the region $|x-y|\geq1$. For the term $I_2$ we write
$$
I_2=\int_{\R^3}dx |w(x)|^2\int_{|x-y|<1} \frac{dy}{|x-y|^{3+2s}}\left|\int_{\R^3}(\nabla u(z))v(z)(|x-z|^{-1}-|y-z|^{-1})dz\right|^2
$$
Notice that one has the estimate, for any $s\in(0,1/2)$, 
$$
(|x-z|^{-1}-|y-z|^{-1})\lesssim \frac{|x-y|^{2s}}{|x-z|^{1+2s}} +  \frac{|x-y|^{2s}}{|y-z|^{1+2s}}.
$$
If $|x-z| \leq |y-z|$, this can be obtained by interpolating the obvious inequality $|(|x-z|^{-1}-|y-z|^{-1})|\leq|x-z|^{-1}$ with the inequality
$$
|(|x-z|^{-1}-|y-z|^{-1})|\leq \frac{|y-x|}{|x-z|^2} .
$$
This is due to
$$
|(|x-z|^{-1}-|y-z|^{-1})|=\frac{|z-y|-|x-z|}{|x-z||x-y|}\leq \frac{|y-x|}{|x-z|^2}
$$
where we have only used triangle inequality. We deal with the case $|z-y| \leq |x-z|$ in the same way

Therefore, we have
$$
I_2\lesssim \int_{\R^3}dx |w(x)|^2\int_{|x-y|<1} \frac{dy}{|x-y|^{3-2s}}\Big( \left|\int_{\R^3}\frac{\nabla u(z)}{|x-z|^s}\frac{v(z)}{|x-z|^{1+s}}dz\right|^2 + \left|\int_{\R^3}\frac{\nabla u(z)}{|y-z|^s}\frac{v(z)}{|y-z|^{s+1}}dz\right|^2 \Big)
$$
$$
\lesssim\|u\|_{H^{1+s}}^2\|v\|_{H^{1+s}}^2\int_{\R^3}dx |w(x)|^2\int_{|x-y|<1} \frac{dy}{|x-y|^{3-2s}}
$$
where we have used again \eqref{genhardy} with $s$ and $1+s < \frac32$. As $s>0$, the integral in $dy$ is finite and then we eventually get
$$
I_2\lesssim \|u\|_{H^{s+1}}^2\|v\|_{H^{s+1}}^2\|w\|_{L^2}^2
$$
and this concludes the proof in the first case.

%We now turn to the second case, i.e. $\sigma\in(1/2,1)$. 
%By using the fractional Leibniz rule we have
%$$
%\| (uv * |x|^{-1})w\|_{H^{1+s}}\leq\| (uv * |x|^{-1})\|_{L^\infty}\|w\|_{H^{1+s}}+\| (uv * |x|^{-1})\|_{W^{1+s,\infty}}\|w\|_{L^2};
%$$
%the first addendum in the right hand side can be estimated by using Lemma \ref{lem-multiL2}, therefore to conclude with we only need to estimate $\| (uv * |x|^{-1})w\|_{W^{1+s,\infty}}$. As clearly
%$$
%\| (uv * |x|^{-1})\|_{W^{1+s,\infty}}\leq \| (uv * |x|^{-1})\|_{W^{2,\infty}},
%$$
%we are reduced to estimate $\|\Delta(uv * |x|^{-1})w\|_{L^\infty}$. We write
%$$
%\|\Delta(uv * |x|^{-1})\|_{L^\infty}\cong\sum_{i=1}^3\|\partial_i\cdot ((\partial_i u v+u\partial_i v) * |x|^{-1})\|_{L^\infty}=
%$$
%$$
%\sum_{i=1}^3\| ((\partial_i u v+u\partial_i v) * \partial_i|x|^{-1})\|_{L^\infty}=\sum_{i=1}^3\| ((\partial_i u v+u\partial_i v) * \frac{x_i}{|x|^3}\|_{L^\infty}.
%$$
%We deal with one of the terms, the other one being analogous by symmetry; we write
%$$
%\| \partial_i u v * \frac{x_i}{|x|^3}\|_{L^\infty}=\int_{\R^3}\frac{x_i}{|x|^3}v(y-x)\partial_iu(y-x)dx
%=\int_{\R^3}\frac{v(y-x)}{|x|^{2-s}}\frac{x_i\partial_iu(y-x)}{|x|^{1+s}}dx\leq
%$$
%$$
%\left\|\frac{v(y-x)}{|x|^{2-s}}\right\|_{L^2}\left\|\frac{\partial_i u(y-x)}{|x|^{s}} \right\|_{L^2}\leq \|v\|_{H^{2-s}}\|u\|_{H^{1+s}}
%$$
%where in the last line we have used \eqref{genhardy}, with $\sigma<1 < \frac32$ and $2-\sigma<\frac32 < 1+s$, and this concludes the proof.
\end{proof}

\subsection{Contraction for $u$ with fixed $q$}

In this subsection, we assume that $C_1$ is the constant defined in Proposition \ref{prop-linH1} and in Proposition \ref{prop-estimH1sn}. We begin by stating a well posedness result for the nonlinear Dirac equation in $H^1$.

\begin{proposition}\label{contrac-u}
Let $\sigma\in[1,3/2)$, and $R>0$.
There exists a constant $C$ such that for all $u_0 \in H^\sigma$, and all $q$ that satisfies
$$
 \sup_{k,t}|\dot q_k(t)|\leq \frac{C_1}{2},\quad \sup_{k} \|\ddot{q}_k\|_{L^1([0,1])} \leq R
$$
the Cauchy problem 
$$
\begin{cases}
\displaystyle
i\partial_t u = (\mathcal{D}+\beta) u + \sum_{k=1}^N \frac{Z_k}{|x-q_k(t)|}u +( |x|^{-1}*|u|^2) u \\
u(0,x) = u_0(x)
\end{cases}
$$
is well posed in $\mathcal C([0,T],H^{\sigma})$ for $T \leq \frac1{C\|u_0\|_{H^{\sigma}}^2},1$ with the additional condition, if $N>1$, that $T \sup_k \|\dot q_k\|_{L^\infty} \leq \frac{C_1}{2}$. Let us denote the corresponding flow by $\Psi_q(t)$. Then $\Psi_q(t)$ is continuous in the initial datum: namely, for each $u_0$, $v_0\in H^1$ we have
$$
\|\Psi_q(t)u_0 - \Psi_q(t)v_0\|_{\mathcal C([0,T],H^{\sigma})} \leq C \|u_0 - v_0\|_{H^{\sigma}}.
$$
\end{proposition}

\begin{proof}
Thanks to the assumptions on $q$ and $T$ in the several nuclei case ($T \sup_k \|\dot q_k\|_{L^\infty} \leq \frac{C_1}{2}$), the Cauchy problem admits the following Duhamel formulation for $t<T$ :
$$
u = U_q(t,0)u_0 -i \int_{0}^t U_q(t,\tau) (|x|^{-1}*|u|^2u) d\tau.
$$
Let
$$
A_q(u) = U_q(t,0)u_0 -i \int_{0}^t U_q(t,\tau) (|x|^{-1}*|u|^2)u d\tau.
$$
We prove that $A_q$ admits a fixed point by a contraction argument.

Thanks to the continuity of $U_q$ on $H^\sigma$ (see Proposition \ref{moreregprop2}) we have 
$$
\|A_q(u)\|_{\mathcal C([0,T],H^\sigma)} \leq C \|u_0\|_{H^\sigma} + CT \|(|x|^{-1}*|u|^2)u\|_{\mathcal C([0,T],H^\sigma)}.
$$
And thanks to the bilinear estimates of Lemma \ref{lem-multiL2bis}
$$
\|A_q(u)\|_{\mathcal C([0,T],H^\sigma)} \leq C \|u_0\|_{H^\sigma} + CT \|u\|^3_{\mathcal C([0,T],H^\sigma)}.
$$
Therefore, the ball of $\mathcal C([0,T],H^\sigma)$ of center $0$ and radius $2C\|u_0\|_{H^\sigma}^2$ is stable under $A_q$ for $T \leq \frac1{C'\|u_0\|_{H^\sigma}^2}$.

What is more,
$$
\|A_q(u) - A_q(v)\|_{\mathcal C([0,T],H^\sigma)} \leq CT \|(|x|^{-1}*|u|^2)u - (|x|^{-1}*|v|^2)v\|_{\mathcal C([0,T],H^\sigma)}
$$
and 
$$
(|x|^{-1}*|u|^2)u - (|x|^{-1}*|v|^2)v = (|x|^{-1}*|u|^2)(u-v) + v|x|^{-1}*(\Re \an{u+v,u-v}).
$$
Hence, thanks to Lemma \ref{lem-multiL2}, we get that for $u,v$ in the ball of $\mathcal C([0,T],H^\sigma)$ of center $0$ and radius $2C\|u_0\|_{H^\sigma}^2$,
$$
\|A_q(u) - A_q(v)\|_{\mathcal C([0,T],H^\sigma)} \leq C_2T \|u_0\|_{H^\sigma}^2\|u-v\|_{\mathcal C([0,T],H^\sigma)}
$$
thus $A_q$ is contracting for $T< \frac1{C_2 \|u_0\|_{H^\sigma}^2}$, and this concludes the proof.
\end{proof}

\subsection{Further properties of $\Psi_q$}

\begin{proposition}\label{prop-flow}  
Let $\sigma\in[1,3/2)$ and $R>0$. There exists a constant $C$ such that for all $u_0 \in H^\sigma$ and $q$ as in the previous proposition, for $T \leq \frac1{C(1+\|u_0\|_{H^\sigma}^2)}$ the flow $\Psi_q$ satisfies the following properties:\\
\item{(i)} $$\|\Psi_q(t)u_0\|_{\mathcal C([0,T],H^{\sigma})} \leq C \|u_0\|_{H^{\sigma}}.$$
\item{\text(ii)} $\Psi_q$ is Lipschitz-continuous in $q$, namely
$$
\|\Psi_{q^{(1)}}(t)u_0 - \Psi_{q^{(2)}}(t)u_0\|_{\mathcal C([0,T],H^{\sigma-1})} \leq C T\|u_0\|_{H^{\sigma}}\sup_{k,t} |\dot q_k^{(1)}(t) - \dot q_k^{(2)}(t)|.
$$

\end{proposition}

\begin{proof}

The property (i) is a consequence of the contraction argument made in the proof of the previous proposition.

For (ii), write $u_j = \Psi_{q^{(j)}}(t) u_0$, $j=1,2$. Since $u_j$ is the fixed point of $A_{q^{(j)}}$ we get
$$
u_1 - u_2 = A_{q^{(1)}}(u_1) - A_{q^{(2)}}(u_2) = I + II + III
$$
with
\begin{eqnarray*}
I &=& U_{q^{(1)}}(t,0) (u_0) -U_{q^{(2)}}(t,0)(u_0) \\
II &=& -i\int_{0}^t (U_{q^{(1)}}(t,\tau) - U_{q^{(2)}}(t,\tau)) (|x|^{-1} * |u_1(\tau)|^2) u_1(\tau) d\tau\\
III &=& -i\int_{0}^t U_{q^{(2)}}(t,\tau)\Big( (|x|^{-1} * |u_1(\tau)|^2) u_1(\tau) -(|x|^{-1} * |u_2(\tau)|^2)u_2(\tau) \Big) d\tau.
\end{eqnarray*}
From Corollary \ref{cor-contUsev}, we have
$$
\|U_{q^{(1)}}(t,\tau) - U_{q^{(2)}}(t,\tau)\|_{H^{\sigma} \rightarrow H^{\sigma-1}}  \leq C \sup_{k} T \|(\dot{q_k}^{(1)})- (\dot{q_k}^{(2)})\|_{L^\infty}.
$$
Therefore, we get
$$
\|I\|_{\mathcal C([0,T],H^\sigma)} \leq C \sup_{k,t} T |(\dot{q_k}^{(1)})- (\dot{q_k}^{(2)})| \|u_0\|_{H^{\sigma+1}}
$$
and
$$
\|II\|_{\mathcal C([0,T],H^\sigma)} \leq C \sup_{k,t} T |(\dot{q_k}^{(1)})- (\dot{q_k}^{(2)})| \|(|x|^{-1} * |u_1|^2 )u_1 \|_{\mathcal C([0,T],H^{\sigma+1})}.
$$
Since 
$$
T \|(|x|^{-1} * |u_1|^2) u_1 \|_{\mathcal C([0,T],H^{\sigma+1})} \leq C' \|u_0\|_{H^{\sigma+1}},
$$
we get
$$
\|II\|_{\mathcal C([0,T],H^1)} \leq C \sup_{k,t} T |(\dot{q_k}^{(1)})- (\dot{q_k}^{(2)})| \|u_0\|_{H^{\sigma+1}}.
$$
We have 
$$
III = A_{q^{(2)}}(u_1) - A_{q^{(2)}}(u_2)
$$
and since $A_{q^{(2)}}$ is contracting, this yields the result.
\end{proof}

\begin{remark}
Actually, by exploiting some refined version of estimates in Lemma \ref{lem-multiL2bis}, the time of existence $T$ could be made dependent only on the $H^1$ norm of the initial datum $u_0$. Anyway, we prefer to keep the strongest assumption with the dependence on $H^\sigma$ because this simplificates significantly the proof of the Lipschitz-continuity, that is point (ii) above.
\end{remark}

\subsection{A Schauder fixed point for $q$}\par\noindent\vskip5pt

Let $t_0\geq 0$, $(a_k)_k \in (\R^3)^N$ and $(b_k)_k \in (\R^3)^N$. 

Assuming for the moment that it is well defined, we set $P(q)$ as the vector field such that $\ddot{P}(q)_k=F(q)_k$ for every $k$ and where
$$
 F(q)_k = -Z_k \an{\Psi_q(u_0)| \frac{q_k -x}{|q_k - x|^3} |\Psi_q(u_0)} + \sum_{l\neq k} Z_k Z_l \frac{q_k-q_l}{|q_k-q_l|^3}
$$
where the following initial values are given: $P(q)_k(t=t_0)_k = a_k$, $\dot{P}_q(t=t_0)_k = b_k$.

Let $R>0$ and 
\begin{multline*}
B = \{ q\in \mathcal C^2([t_0,t_0+T_1],(\R^3)^N) | T_1\sup_k \|\dot q_k\|_{L^\infty} \leq \frac{C_1}{2}, \\
 \|\dot q_k\|_{L^\infty}\leq \frac{C_1}{2}, \sup_k \|\ddot q_k\|_{L^1}\leq R, q_k(t_0) = a_k, \forall k\neq l, \forall t |q_k - q_l|> 4\varepsilon_0\}
\end{multline*}
where $M_0$, $C_1$ are the constants fixed by Proposition \ref{prop-estimH1sn}.

\begin{proposition}\label{stableprop}
 Assume $u_0\in H^1$, $\sum_k |b_k|\leq \frac{C_1}{4}$ and for all $k\neq l$, $|a_k-a_l|\geq 8\varepsilon_0$. There exists $C$ such that if $T_1\leq \frac1{C\|u_0\|_{H^1}^2}$ for $N=1$ and $T_1\leq \frac1{C(1+\|u_0\|_{H^1}^2)}$ if $N\geq 2$, then $B$ is stable under $P$. 
\end{proposition}

\begin{proof}
We have 
$$
\|F(q)_k\|_{L^\infty} \leq C |Z_k|\|\Psi_q(t)u_0\|_{H^1}^2 +C \varepsilon_0^{-2} \sum_{k\neq l} |Z_kZ_l |.
$$
Hence
$$
\| F(q)_k\|_{L^1}\leq T_1 C\Big( \|u_0\|_{H^1}^2 + 1_{N\geq 2}\Big).
$$
This yields 
$$
\|\dot{P}(q)_k\|_{L^\infty} \leq |b_k|+ T_1 C\Big( \|u_0\|_{H^1}^2 + 1_{N\geq 2}\Big)
$$
hence for $T_1 \leq C_2  C^{-1}\Big( \|u_0\|_{H^1}^2 + 1_{N\geq 2}\Big)^{-1}$ and $T_1\leq N^{-1}\frac{C_1}{4}C^{-1}\Big( \|u_0\|_{H^1}^2 + 1_{N\geq 2}\Big)^{-1}$, we have 
$$
\| F(q)_k\|_{L^1}\leq C_2
$$
and
$$
 \|\dot{P}(q)_k\|_{L^\infty} \leq \frac{C_1}{2}.
$$
For $N\geq 2$, we have
$$
T_1 \|\dot{P}(q)_k\|_{L^\infty} \leq T_1 \frac{C_1}{2}
$$
hence for $T_1 \leq 1$, we have 
$$
T_1 \|\dot{P}(q)_k\|_{L^\infty} \leq \frac{C_1}{2}.
$$
Finally, for $k\neq l$, we have 
$$
|q_k - q_l|\geq |a_k - a_k| - |q_k -a_k| - |q_l-a_l|
$$
and since $|q_l - a_l|\leq C_1 T_1$, we get that for $T_1 \leq \frac{2\varepsilon_0}{C_1}$,
$$
|q_k -q_l|\geq 4\varepsilon_0.
$$

Hence $B$ is stable under $P$.
\end{proof}

In the next proposition, we give the key properties of the map $P$, that is we prove that it is H\"older continuous if $u_0\in H^\sigma$ for $\sigma\in(1,\frac32)$. The threshold is a consequence of the threshold for the validity of Hardy inequality \eqref{genhardy}, that will play a key role in the proof.

\begin{proposition}\label{regp}
Let $q^{(j)} \in B$ for $j=1,2$. Then
%\begin{itemize}
%\item
%if $\sigma\in(3/2,2]$ there exists $C$ and $\tilde{T}\leq T_1$ with $T_1$ as in Proposition \ref{stableprop} such that for all $u_0 \in H^\sigma$ we have
%\begin{equation}\label{lipsch}
%\sup_k \|P(q^{(1)})_k - P(q^{(2)})_k\|_{\mathcal C^1([t_0,t_0+\tilde{T}])} \leq C \tilde{T}^3(1_{N\geq 2}+ \|u_0\|_{H^\sigma}^2) \sup_k \|\dot q^{(1)}_k - \dot q^{(2)}_k\|_{L^\infty}.
%\end{equation}
%\item 
if $\sigma\in(1,3/2)$ there exists $C$ and $\tilde{T}\leq T_1$ with $T_1$ as in Proposition \ref{stableprop} such that for all $u_0 \in H^\sigma$ we have
\begin{equation}\label{hold}
\sup_k \|P(q^{(1)})_k - P(q^{(2)})_k\|_{\mathcal C^1([t_0,t_0+\tilde T])} \leq C \tilde T^{2s}(1_{N\geq 2}+ \|u_0\|_{H^\sigma}^2) \sup_k \|\dot q^{(1)}_k - \dot q^{(2)}_k\|_{L^\infty}^{2s-2}.
\end{equation}
%\end{itemize}
\end{proposition}

\begin{proof} 
First of all we write

\begin{multline*}
 F(q^{(1)})_k -  F(q^{(2)})_k = -Z_k \an{\Psi_{q^{(1)}}(u_0)|\frac{q_k^{(1)}-x}{|q_k^{(1)}-x|^3} |\Psi_{q^{(1)}}(u_0)}\\
+Z_k \an{\Psi_{q^{(2)}}(u_0)|\frac{q_k^{(2)}-x}{|q_k^{(2)}-x|^3}| \Psi_{q^{(2)}}(u_0)}\\
+ 2\sum_{k\neq l} Z_kZ_l \Big(\frac{q_k^{(1)}-q_l^{(1)}}{|q_k^{(1)}-q_l^{(1)}|^3} -\frac{q_k^{(2)}-q_l^{(2)}}{|q_k^{(2)}-q_l^{(2)}|^3}\Big).
\end{multline*}

Notice that, since $|q_k^{(j)}-q_l^{(j)}|\geq 2\varepsilon_0$, we have 
$$
\Big| \frac{q_k^{(1)}-q_l^{(1)}}{|q_k^{(1)}-q_l^{(1)}|^3} -\frac{q_k^{(2)}-q_l^{(2)}}{|q_k^{(2)}-q_l^{(2)}|^3}\Big| \lesssim \varepsilon_0^{-3}(|q_k^{(1)}  - q^{(2)}_k|+|q_l^{(1)}  - q^{(2)}_l|),
$$
so that we get
$$
2\sum_{k\neq l} Z_kZ_l \Big( \frac{q_k^{(1)}-q_l^{(1)}}{|q_k^{(1)}-q_l^{(1)}|^3} -\frac{q_k^{(2)}-q_l^{(2)}}{|q_k^{(2)}-q_l^{(2)}|^3}\Big)\leq C \sum_{k\neq l} |Z_kZ_l| T \varepsilon_0^{-3} \sup_k \|\dot q^{(1)}_k - \dot q^{(2)}_k\|_{L^\infty}.
$$

For the other part of the difference, we have 
$$
\an{\Psi_{q^{(1)}}(u_0)|\frac{q_k^{(1)}-x}{|q_k^{(1)}-x|^3} |\Psi_{q^{(1)}}(u_0)}-\an{\Psi_{q^{(2)}}(u_0)|\frac{q_k^{(2)}-x}{|q_k^{(2)}-x|^3} |\Psi_{q^{(2)}}(u_0)} = I + II +III
$$
with
\begin{eqnarray*}
I &=& \an{\Psi_{q^{(1)}}(u_0)- \Psi_{q^{(2)}}(u_0)|\frac{q_k^{(1)}-x}{|q_k^{(1)}-x|^3} |\Psi_{q^{(1)}}(u_0)} \\
II &=& \an{\Psi_{q^{(2)}}(u_0)|\Big( \frac{q_k^{(1)}-x}{|q_k^{(1)}-x|^3} -\frac{q_k^{(2)}-x}{|q_k^{(2)}-x|^3}\Big)| \Psi_{q^{(1)}}(u_0)} \\
III &=& \an{\Psi_{q^{(2)}}(u_0)|\frac{q_k^{(2)}-x}{|q_k^{(2)}-x|^3}|\Big( \Psi_{q^{(1)}}(u_0)- \Psi_{q^{(2)}}(u_0) \Big)}
\end{eqnarray*}

We deal with the three terms separately. For the term $I$ (and in fact $III$) we can write, with $a\in (0,1)$, and $b\in \R$

\begin{eqnarray*}
|I| &\leq& \int_{\R^3}\frac{|\Psi_{q^{(1)}}(u_0)- \Psi_{q^{(2)}}(u_0)||\Psi_{q^{(1)}}(u_0)|}{|q_k^{(1)}-x|^2}
\\
&=&
\int_{\R^3}
\frac{|\Psi_{q^{(1)}}(u_0)|}{|q_k^{(1)}-x|^{\sigma}}\frac{|\Psi_{q^{(1)}}(u_0)- \Psi_{q^{(2)}}(u_0)|^{a}}{|q_k^{(1)}-x|^{b}}
\frac{|\Psi_{q^{(1)}}(u_0)- \Psi_{q^{(2)}}(u_0)|^{1-a}}{|q_k^{(1)}-x|^{2-\sigma-b}}
\end{eqnarray*}
and by H\"older's inequality,
$$
|I|
\leq C
\left\|\frac{\Psi_{q^{(1)}}(u_0)}{|q_k^{(1)}-x|^{\sigma}}\right\|_{L^2}
\left\|\frac{(\Psi_{q^{(1)}}(u_0)- \Psi_{q^{(2)}}(u_0))^{a}}{|q_k^{(1)}-x|^{b}}\right\|_{L^{2/a}}
\left\|\frac{(\Psi_{q^{(1)}}(u_0)- \Psi_{q^{(2)}}(u_0))^{1-a}}{|q_k^{(1)}-x|^{2-\sigma-b}}\right\|_{L^{2/(1-a)}},
$$
which identifies as
$$
|I|
\leq C
\left\|\frac{\Psi_{q^{(1)}}(u_0)}{|q_k^{(1)}-x|^{\sigma}}\right\|_{L^2}
\left\|\frac{(\Psi_{q^{(1)}}(u_0)- \Psi_{q^{(2)}}(u_0)}{|q_k^{(1)}-x|^{b/a}}\right\|_{L^{2}}^a
\left\|\frac{(\Psi_{q^{(1)}}(u_0)- \Psi_{q^{(2)}}(u_0))}{|q_k^{(1)}-x|^{(2-\sigma-b)/(1-a)}}\right\|_{L^{2}}^{1-a}.
$$

To conclude we need
$$
\begin{cases}
\frac{b}{a}<\sigma-1\\
\frac{2-\sigma-b}{1-a}\leq \sigma.
\end{cases}
$$
Taking the equality in the last inequalies, we get $a = 2(\sigma-1)$, and this justifies the condition $\sigma\in (1,3/2) \Leftrightarrow a\in (0,1)$, and $b = 2(\sigma-1)^2$.

Therefore one gets, again by  \eqref{genhardy} with $s$ and $s_1$ and Proposition \ref{prop-flow}, 
\begin{eqnarray*}
|I| &\lesssim & \|\Psi_{q^{(1)}}(u_0)\|_{H^\sigma}(\|\Psi_{q^{(1)}}(u_0)\|_{H^\sigma}+\|\Psi_{q^{(2)}}(u_0)\|_{H^\sigma})^{1-a}\|\Psi_{q^{(1)}}(u_0)- \Psi_{q^{(2)}}(u_0)\|_{H^{\sigma-1}}^{a}
\end{eqnarray*}
which ensures 
$$
|I| \lesssim \|u_0\|_{H^\sigma}^2 \tilde T^a \sup_k \|\dot q_k^{(1)} - \dot q_k^{(2)}\|_{L^\infty}.
$$

To deal with the term $II$ we have, for $\alpha \in (0,1)$
\begin{eqnarray*}
|II| &\leq& \int_{\R^3} |\Psi_{q^{(1)}}(u_0)||\Psi_{q^{(2)}}(u_0)|\left(\frac{|q_k^{(1)}-q_k^{(2)}|^{\alpha}}{|q_k^{(1)}-x|^{2+\alpha}}+\frac{|q_k^{(1)}-q_k^{(2)}|^{\alpha}}{|q_k^{(1)}-x|^{2+\alpha}}\right)
\\
&\leq&C
\nonumber
|q_k^{(1)}-q_k^{(2)}|^{\alpha}\left\| \frac{\Psi_{q^{(1)}}(u_0)}{|x|^{1+\alpha/2}}\right\|_{L^2}\left\| \frac{\Psi_{q^{(2)}}(u_0)}{|x|^{1+\alpha/2}}\right\|_{L^2}
\\
\nonumber
&\leq&C
|q_k^{(1)}-q_k^{(2)}|^{\alpha}
\left\|\Psi_{q^{(1)}}(u_0)\right\|_{H^{1+\alpha/2}}\left\|\Psi_{q^{(2)}}(u_0)\right\|_{H^{1+\alpha/2}}
\\
\nonumber
&\leq&C
|q_k^{(1)}-q_k^{(2)}|^{\alpha}\| u_0\|_{H^{1+\alpha/2}}^2.
\end{eqnarray*}
Taking then $\alpha=2\sigma-2$ and using Proposition \ref{prop-flow} we get in the end
$$
|II| \leq C\| u_0\|_{H^\sigma}^2  \sup_k \| q^{(1)}_k -  q^{(2)}_k\|_{L^\infty}^{2s-2}
\leq  CT \| u_0\|_{H^\sigma}^2  \sup_k \|\dot q^{(1)}_k - \dot q^{(2)}_k\|_{L^\infty}^{2s-2};
$$
note that $(2\sigma-2)\in(0,1)$ for $\sigma\in(1,3/2)$, so that we obtain the H\"olderianity of this term, and the proof is concluded
\end{proof}

%Due to the different nature (resp. H\"older and Lipschtiz continuous, as seen) of the map $F(q)_k$ depending on the regularity of the initial condition $u_0$, we thus give two different results for the existence of solutions for the differential equation $m_k\ddot{q}_k=  F(q)_k$, depending on the regularity of $u_0$. 
%For the {\em existence} of a solution, we have the following

As a consequence, we are able to prove the following result.

\begin{proposition}\label{prop-schauder} 
Let $\sigma\in[1,3/2)$. There exists $C=C(m)$ such that for all $(a_k)$ and $(b_k)$ such that for all $l\neq k$, $|a_k-a_l|\geq 8\varepsilon_0$, $(b_k)_k\leq \frac{C_1}{4}$, for all $u_0\in H^\sigma$, the system of equations
$$
m_k\ddot{q}_k=  F(q)_k
$$
with initial data $q_k(t=0) = a_k$ and $\dot q_k(t=0)=b_k$ admits a solution in $\mathcal C^2([0,T]$ for $T \leq  \frac1{C (1_{N\geq 2}+\|u_0\|_{H^\sigma}^2)}$.
\end{proposition}

\begin{proof} This is a Schauder fixed point argument for $P$ in
\begin{multline*}
B(0,T) = \{ q\in \mathcal C^2([0,T],(\R^3)^N) | T\sup_k \|\dot q_k\|_{L^\infty} \leq M_0, \\
\sum_k \|\dot q_k\|_{L^\infty}\leq C_1, \sup_k \|\ddot q_k\|_{L^1}\leq R, q_k(0) = a_k, \forall k\neq l, \forall t |q_k - q_l|> 4\varepsilon_0\}
\end{multline*}
for the topology of $\mathcal C^1([0,T])$,  as a consequence of Proposition \ref{stableprop} and \eqref{hold} in Proposition \ref{regp}

Note that $B(0,T)$ is compact in $\mathcal C^1([0,T])$ because of the boundedness of the $\mathcal C^2$ norm.

\end{proof}

\vskip20pt
\section{Appendix: a remark on regularity of the ground state of Dirac-Coulomb Hamiltonian}
Let us consider the Sobolev regularity of Dirac-Coulomb eigenstates when $Z<\sqrt3/2$.\\ In particular, we are interested in the ground state.
Its generic component has the form (see \cite{LL-4, GalMich18})

\begin{equation}
f(r)=\text{const}\times e^{-a r} r^{b-1}\ 
\end{equation}

where
\begin{equation} a=Z\ ,\ \ \ \ \ b=\sqrt{1-(\alpha Z)^2} \equiv \sqrt{1- \nu^2},\ \ \ \ \nu\in (0,1)
\end{equation}
The Fourier transform of the above radial function satisfies

$$
\hat f(k)=\frac{2}{\sqrt2\pi}\frac{1}{k}\int_0^{\infty}\ rf(r)\sin(kr)\ dr
$$
The above integral is the Fourier sine-transform of $rf(r)=\text{const}\times e^{-a r} r^b$.
\vskip5pt
According to the Table of Integral Transforms I, {\it Bateman Manuscript Project}, formula 2.4 (7) pag 72 in \cite{Bateman54}, one has
\begin{align}
\hat f(k) =&\frac{2}{\sqrt2\pi}\frac{1}{k} \Gamma(b +1)(a^2+k^2)^{-\frac{1}{2}(b+1)} \sin[(b+1)\tan^{-1}(\frac{k}{a})]\\=&\text{const}\times\frac{1}{k} (a^2+k^2)^{-\frac{1}{2}(b+1)} \sin[(b+1)\tan^{-1}(\frac{k}{a})]
\end{align}
This is regular at the origin while the asymptotic behavior at infinity is given by
\begin{equation}
\hat f \sim k^{-(b+2)}
\end{equation}
Now $f\in H^\sigma\iff |\hat f|^2(1+k^2)^s$ is integrable in $\R^3$, from which we get the condition $f\in H^\sigma\iff k^{(-2b-4)}(1+k^2)^sk^2$ is integrable at infinity, i.e.
\begin{equation}
2b+2-2\sigma>1 \ \ \ \iff \ \ \ \sqrt{1-\nu^2}>s-\frac{1}{2}\ \ 
\end{equation}
This implies that 
\begin{enumerate}
\item $f\in H^1\ \  \forall \:\nu \in (0,\frac{\sqrt{3}}{2})$\\
\item $f\notin H^{\frac{3}{2}}\ \ \text{whatever}\ \ \ \nu\in (0,\frac{\sqrt{3}}{2})$\\
\item $f\in H^\sigma\ \ \ \sigma=\sigma(\nu)=\frac{3}{2}-\epsilon$ with $\nu^2<2\epsilon-\epsilon^2$ 
\end{enumerate}
So the regularity is better and better (but lower than $H^{3/2}$) with the decreasing of the charge Z.

\end{document}